\documentclass[a4paper]{amsart}
\usepackage{fullpage}
\usepackage{hyperref}
 \hypersetup{
     colorlinks=true,
     linkcolor=purple,
     filecolor=cyan,
     citecolor = cyan,
     urlcolor=cyan,
     }

\usepackage{scalerel}
\usepackage{tikz}
\usetikzlibrary{cd}
\usepackage{bbm}
\usepackage[
    style=alphabetic,
    backend=bibtex,
    maxalphanames=99,
    maxnames=99
]{biblatex}
\newcommand{\isomto}{\stackrel{\textstyle\sim}{\smash{\longrightarrow}\rule{0pt}{0.4ex}}}

\DeclareMathOperator{\DA}{\mathrm{DA}}
\newcommand{\ula}{^{\!\mathrm{ULA}}}
\newcommand{\rig}{^{\mathrm{rig}}}
\newcommand{\cons}{_{\mathrm{c}}}
\newcommand{\Divtil}{\widetilde{\mathrm{Div}}_0^{\perp}}
\newcommand{\Divtilk}{\widetilde{\mathrm{Div}}_{0,k}^{\perp}}

\newcommand{\Hk}[1]{\ensuremath{\operatorname{Hk}_{#1}}}

\makeatletter
\def\lboxtimes{\@ifnextchar_{\@lboxtimessub}{\@lboxtimesnosub}}
\def\@lboxtimessub_#1{\mathchoice{\mathbin{\mathop{\boxtimes}^L}_{#1}}%
  {\boxtimes^L_{#1}}{\boxtimes^L_{#1}}{\boxtimes^L_{#1}}}
\def\@lboxtimesnosub{\mathbin{\mathop{\boxtimes}^L}}
\makeatother
\usepackage{amscd}
\usepackage[arrow, matrix, curve]{xy}
\usepackage{color}
\usepackage{blindtext}
\usepackage{tikz}
\usepackage{tikz-cd}
\usepackage{comment}
\usepackage{mathrsfs}
\usepackage[compress,capitalize,nameinlink]{cleveref}

\usepackage{relsize}
\usepackage[bbgreekl]{mathbbol}
\usepackage{amsfonts}
\usepackage[T1]{fontenc}
\DeclareSymbolFontAlphabet{\mathbb}{AMSb}
\DeclareSymbolFontAlphabet{\mathbbl}{bbold}

\numberwithin{equation}{section}
\newtheorem{Satz}{Satz}[section]
\theoremstyle{definition}
\newtheorem{Corollary}[Satz]{Corollary}
\newtheorem{Lemma}[Satz]{Lemma}
\newtheorem{Definition}[Satz]{Definition}
\newtheorem{Remark}[Satz]{Remark}
\newtheorem{Theorem}[Satz]{Theorem}

\newtheorem{proposition}[Satz]{Proposition}

\newtheorem{Situation}[Satz]{Situation}

\usepackage{amssymb}
\usepackage{fullpage}

\DeclareMathOperator{\Hom}{Hom}

\DeclareMathOperator{\Aut}{Aut}
\DeclareMathOperator{\Gal}{Gal}

\DeclareMathOperator{\Lie}{Lie}
\DeclareMathOperator{\Spec}{Spec}

\DeclareMathOperator{\GL}{GL}

\DeclareMathOperator{\Ker}{Ker}

\DeclareMathOperator{\ad}{ad}
\DeclareMathOperator{\eq}{eq}

\DeclareMathOperator{\G}{\mathcal{G}}

\DeclareMathOperator{\Z}{\mathbb{Z}}
\DeclareMathOperator{\colim}{colim}

\DeclareMathOperator{\Pic}{Pic}

\DeclareMathOperator{\Set}{Set}
\DeclareMathOperator{\perf}{perf}

\DeclareMathOperator{\Res}{Res}

\DeclareMathOperator{\red}{red}

\DeclareMathOperator{\Mod}{Mod}

\DeclareMathOperator{\op}{op}
\DeclareMathOperator{\Gr}{Gr}
\DeclareMathOperator{\Witt}{Witt}
\DeclareMathOperator{\Perf}{Perf}

\DeclareMathOperator{\Kos}{Kos}

\DeclareMathOperator{\D}{D}

\DeclareMathOperator{\Stab}{Stab}

\DeclareMathOperator{\Cent}{Cent}

\DeclareMathOperator{\Ind}{Ind}

\DeclareMathOperator{\Cat}{\mathrm{Cat}}

\newcommand{\st}{\mathrm{st}}
\newcommand{\Prst}{\mathrm{Pr}^\st}
\DeclareMathOperator{\PZ}{PZ}
\DeclareMathOperator{\Sat}{Sat}

\DeclareMathOperator{\Int}{Int}
\DeclareMathOperator{\bd}{bd}
\DeclareMathOperator{\CT}{CT}
\DeclareMathOperator{\naive}{naive}
\DeclareMathOperator{\abs}{abs}
\DeclareMathOperator{\corr}{corr}

\newcommand{\et}{\operatorname{\acute{e}t}}

\newcommand{\mbF}{\mathbb{F}}

\newcommand{\mbZ}{\mathbb{Z}}

\newcommand{\mcA}{\mathcal{A}}
\newcommand{\mcB}{\mathcal{B}}
\newcommand{\mcC}{\mathcal{C}}

\newcommand{\mcF}{\mathcal{F}}
\newcommand{\mcG}{\mathcal{G}}
\newcommand{\mcH}{\mathcal{H}}

\newcommand{\mcN}{\mathcal{N}}
\newcommand{\mcO}{\mathcal{O}}
\newcommand{\mcP}{\mathcal{P}}

\newcommand{\mcS}{\mathcal{S}}
\newcommand{\mcT}{\mathcal{T}}
\newcommand{\mcU}{\mathcal{U}}
\newcommand{\mcV}{\mathcal{V}}

\newcommand{\mfa}{\mathfrak{a}}

\newcommand{\id}{\mathrm{id}}

\newcommand{\ami}{\color{purple}}

\author{Sebastian Bartling}
\address{Universität Duisburg-Essen, Fakultät für Mathematik, Thea-Leymann Straße, 45127 Essen, Germany}
\email{sebastian-bartling@hotmail.de}
\author{R{\i}zacan \c{C}ilo\u{g}lu}
\address{Universität Duisburg-Essen, Fakultät für Mathematik, Thea-Leymann Straße, 45127 Essen, Germany}
\email{rizacan.ciloglu@gmail.com}
\title{Parahoric motivic Hecke categories in equal and mixed characteristic}
\title{Parahoric motivic Hecke categories in equal and mixed characteristic}
\date{\today}
\usepackage{silence}
\ActivateWarningFilters[todo]
\begin{document}
\begin{abstract}
    We study constructible $\mathbb{Z}[\tfrac{1}{p}]$-linear \'etale motives
    on parahoric Hecke stacks associated with quasi-split tamely ramified reductive groups.
    We prove a canonical equivalence between their
    equal-characteristic and Witt-vector incarnations that is monoidal with respect to convolution. This extends Bando's comparison
    \cite{bandoComparison}, which concerns split
    reductive groups and constructible \'etale sheaves. The proof combines the study of
    a family of affine Grassmannians for Pappas--Zhu group
    schemes with the use of the motivic nearby cycles functor.
\end{abstract}
\maketitle

\section{Introduction}
Parahoric affine flag varieties, and the Hecke categories built from them,
have parallel incarnations in equal and mixed characteristic. Over a local
function field they arise from power-series loop groups, whereas over a
$p$-adic field the corresponding construction uses ramified Witt vectors.
Their Schubert strata are governed by the same Iwahori--Weyl group and the
same Bruhat order, but this common combinatorics does not by itself compare
the ambient spaces or their sheaf theories. For split reductive groups, Bando
\cite{bandoComparison} has found an ingenious geometric construction bridging this divide by placing the two affine
Grassmannians in a single family in perfect geometry and comparing the
associated categories of \'etale sheaves.

There are at least two related reasons to seek a similar comparison before passing to an
$\ell$-adic realization. A $\Z[\tfrac{1}{p}]$-linear motivic equivalence
controls all $\ell\neq p$-adic realizations at once and remains available for
constructions made at the motivic level. Moreover, motivic Satake
equivalences are now known in equal and mixed characteristic
\cite{richarzscholbachMotivicSatake,casshovesholbachIntegralMotivic,vandenHoveRamifiedMotivicSatake},
and the motivic geometrization of the local Langlands correspondence shows
that working motivically can have arithmetic consequences, such as
independence of $\ell$ for the resulting $L$-parameters
\cite{GeometrizationMotivically}. Therefore, a comparison of motivic sheaves on the two incarnations of Hecke stacks would make it possible to transport motivic results known in equal characteristic directly to the mixed characteristic case.
For example, we expect that our results could be useful to approach a motivic enhancement of Zhu's work \cite{ZhuTameCLLC} on a tame categorical local Langlands correspondence. Namely, Zhu's approach uses Bezrukavnikov's equivalence \cite{bezrukavnikovEquivalence} crucially and therefore one would like to transport a putative motivic version of Bezrukavnikov's equivalence for tamely ramified groups from equal to mixed characteristic.

The purpose of this paper is to establish such a comparison for every
connected reductive group that is quasi-split and splits over a tamely ramified extension and for arbitrary parahoric level.
Neither extension of Bando's result is formal. The ramified case
requires uniform control of Pappas--Zhu group schemes throughout Bando's
family, whereas the extension to motivic sheaves uses the motivic nearby cycles functor without ever relying on the conjectural conservativity of that functor.

Let $F$ be a finite extension of $\mathbb{Q}_p$, with ring of integers
$\mcO_F$, residue field $k_F$, uniformizer
$\pi$ and let $k=\overline{k}_{F}$. Let $G$ be a connected reductive $F$-group that splits over a tamely
ramified extension. Let us choose a facet
$\mcF$ in the Bruhat--Tits building of $G$ over $F$ which gives rise to a parahoric model $\mcP$ of $G$ and a rigidification of $G$.
To these data, one can associate a tamely ramified connected reductive group $G^{\flat}$ over $k_{F}(\!(t)\!)$, a facet $\mcF^{\flat}$ in the Bruhat--Tits building of $G^{\flat}$ over $k_{F}(\!(t)\!)$ and therefore also a parahoric model $\mcP^{\flat}$ of $G^{\flat}$.
One obtains the Hecke stacks
$\Hk{}^{\mathrm{Witt}}$ and $\Hk{}^{\mathrm{eq}}$ in mixed and equal
characteristic and their geometric versions $\Hk{k}^{\mathrm{Witt}}$ resp. $\Hk{k}^{\mathrm{eq}}$. Our main result is the following.
\begin{Theorem}\label{Theorem: Introduction}
    There is a canonical equivalence of stable $\infty$-categories
    $$
    \DA\cons(\Hk{k}^{\mathrm{eq}},\Z[\tfrac{1}{p}])
    \simeq
    \DA\cons(\Hk{k}^{\mathrm{Witt}},\Z[\tfrac{1}{p}])
    $$
    that is monoidal for the convolution structure on both sides.
\end{Theorem}
Here $\DA\cons$ denotes the category of constructible \'etale motives with
bounded support on the corresponding perfect Hecke stack. These
categories are defined in Section \ref{Section: Hecke stack} from finite-level
quotient stacks using the extension of $h$-motives to perfect prestacks
\cite[Appendix A.1]{vandenHoveSplittingModels}. As the geometric objects in Theorem \ref{Theorem: Introduction} are defined over fields, they agree with the
bounded ULA categories as shown by Preis in
\cite[Proposition 3.2.21]{preisUlaMotives}. The underlying finite-type theory
of \'etale motives is due to Ayoub \cite{AyoubEtaleRealizationENS}.
The equivalence is canonical for the fixed Pappas--Zhu data and is established without recourse to any realization.

As mentioned already, Theorem \ref{Theorem: Introduction} extends the split \'etale comparison of
Bando \cite[Theorem 1.1]{bandoComparison} in both the ramified and the motivic directions. Its
compatibility with \'etale realization gives the following consequence.
\begin{proposition}[\'Etale realization and Satake comparison]
\label{Proposition: Introduction Satake comparison}
    Let $\ell\ne p$ and let $x_{0}$ be a special vertex in the Bruhat--Tits building of $G$ over $\breve{F}$. The $\Z_\ell$-adic realization of
    the equivalence in Theorem \ref{Theorem: Introduction} induces a monoidal equivalence
    \[
        \bigl(\Sat(\Hk{k}^{\mathrm{eq}},\Z_\ell),\star\bigr)
        \simeq
        \bigl(\Sat(\Hk{k}^{\mathrm{Witt}},\Z_\ell),\star\bigr),
    \]
    which is symmetrical monoidal for the symmetry constraints on both sides imported from the Satake equivalences of \cite{ModularRamifiedSatake} and \cite{vandenHoveRamifiedMotivicSatake}.
\end{proposition}
This comparison of Satake categories extends \cite[Theorems 1.5 and
1.6]{bandoComparison} from split to quasi-split tamely ramified groups.
\begin{Remark}
    Let us comment on the group theoretical hypothesis. 
    Our result is proven for Hecke stacks over an algebraically closed field and this forces the group used to construct these stacks to be quasi-split.
    We think it is possible that a finer analysis of our argument could descend our results down to $k_{F}$ and therefore avoid the restriction to quasi-split groups.
    Tame ramification however is a more serious restriction because for wildly ramified groups we do not know how to construct an extension of the $F$-group to a reductive group scheme over $\mcO_{F}[u^{\pm}]$ that can be used to define the Bando Grassmannian (cf.\cite{JoaoAffGrassZ}).
\end{Remark}

\begin{Remark}
    While finalizing this paper we became aware of the paper \cite{YunZhu} by Yun--Zhu in which they prove the analogue of our main Theorem \ref{Theorem: Introduction} for étale sheaves by a different method.
    An extension to étale motives, however, is not discussed in their paper. 
\end{Remark}

We now explain how the geometry of Bando's family and the use of the motivic nearby cycles functor combine to prove Theorem
\ref{Theorem: Introduction}.

\subsection{Method of proof}
Let $\mcG$ be the group scheme constructed by Pappas--Zhu in \cite[Theorem 3.1]{PappasZhuLocalModels} starting from the data of $G$, the facet $\mcF$ in the Bruhat--Tits building $B(G;F)$ and the rigidification of $G$. The scheme $\mcG$ is a smooth affine group scheme
with connected fibers over $\mcO_F[u]$ that interpolates between the parahoric models $\mcP$ and $\mcP^{\flat}$ from above. Bando's construction
\cite{bandoComparison} yields a well-behaved family of affine Grassmannians of $\mcG$
\[
    \Gr^B_{\mcG}\longrightarrow\Divtil.
\]
The space $\Divtil$ parametrizes certain pairs $(\xi_{1},\xi_{2})$ of elements that form a Koszul regular sequence in $W_{\mcO_{F}}(R)[\![t]\!]$ and which cut out $\Spec(R)$ inside $\Spec(W_{\mcO_{F}}(R)[\![t]\!])$.
Taking fibers of $\Gr^B_{\mcG}$ at the distinguished
points $(\pi,t)$ and $(t,\pi)$ gives the equal-characteristic and Witt-vector
affine Grassmannians, respectively. Taking the quotient by $L^+\mcG$ gives a
family
\[
    \Hk{}=[L^+\mcG\backslash\Gr^B_{\mcG}]
    \longrightarrow\Divtil
\]
whose geometric distinguished fibers are the two Hecke stacks in Theorem
\ref{Theorem: Introduction}.

The proof of Theorem \ref{Theorem: Introduction} has two ingredients. The first is a study of the geometry of the Bando affine Grassmannian for Pappas--Zhu group schemes. This requires bounds by finite unions of Schubert varieties that are proper over the base, together with a description of each stratum as a classifying stack of a smooth group over the base.
Representability and ind-properness are proved in Proposition
\ref{Proposition: representability Bando affine Grassmannian} and Lemma
\ref{Lemma: Bando affine Grassmannian ind-projective for Pappas--Zhu group schemes}.
The orbit decomposition and the description of the stabilizers are Lemmas
\ref{Lemma: Iwahori Weyl group and double orbits} and
\ref{lemma: structure of Orbits in Bando affine Grassmannian}.
Let us note here that these constructions (stratification into Schubert strata and description of open strata) and properties (properness and smoothness of open strata) naturally hold \emph{over} $\Divtil$; this explains why the Bando affine Grassmannian is indeed a well-behaved family. 

The second ingredient is the replacement of the étale contractibility of spectra of strictly henselian rings, essential to Bando's arguments, by motivic nearby cycles and its compatibilities with ULA motives.
For this, we consider the path
\[
    X_e=D((1-x)^{2e}+x^{2e})\subset\mathbb A^{1,\perf}_{k},
\]
with the morphism to $\Divtilk$ given by $x\mapsto ([1-x]^{e} \pi - [x]^{e}t, [x]^{e}\pi+[1-x]^{e} t)$. 
Here $e$ is the ramification index of a tame splitting field $\widetilde{F}/F$ of $G$; in particular, the choice of the path depends on the group $G$. The endpoints
$0,1:\Spec k\to X_e$ of the path in $\Divtilk$ are $(\pi,t)$ resp. $(-t,\pi)$\footnote{The $-$ in front of $t$ is a curious artefact of the proof but the fiber over it is still the Witt vector Hecke stack.}, respectively. The
reflection $x\mapsto 1-x$ exchanges the points $0$ and $1$. 
Let $S_{0}$ resp. $S_{1}$ be the perfected strict henselizations of $X_e$ at $0$ resp. $1$ and let us denote by $S_{0,\eta}$ and $S_{1,\eta}$ the base change to the generic point $\eta$ of $X_e$.
Then we employ the above path to prove a generic equivalence 
$$
\rho_{\eta}\colon \DA{\ula}(\Hk{S_{0,\eta}})\simeq \DA{\ula}(\Hk{S_{1,\eta}})
$$
which is monoidal for the convolution product.
This is Proposition \ref{prop: generic equivalence}.
This generic equivalence serves as the stepping stone to prove the desired equivalence in Theorem \ref{Theorem: Introduction} using properties of motivic nearby cycles and ULA objects.
This is Theorem \ref{Theorem: Main}. Let us note however that extra care has to be taken to make our equivalence monoidal (even on the $\infty$-categorical level) and this is worked out in \S \ref{ss: monoidality}.
\subsection{Acknowledgments}
While working on this project, S. Bartling was part of the DFG RTG 2553.
The authors thank Thibaud van den Hove for very helpful comments on a draft of this article.
R{\i}zacan \c{C}ilo\u{g}lu was funded by the DFG project 541511946 and acknowledges support from DFG RTG 2553.
He wishes to thank Can Yaylali and Simon Pepin Lehalleur for fruitful discussion about this work.
\subsection{AI Disclosure}
Generative AI was used to edit this text. 
Furthermore, we let ChatGPT 5.6 Sol proof-read a previous version of our article and it brought an error to our attention. In a second version, Claude Opus 5 pointed out a further error and, in the same response, noted an isomorphism similar to the one eventually appearing in Lemma \ref{lem: comparing Bando rings in the generic equivalence}(a). The authors then realized that Lemma \ref{lem: comparing Bando rings in the generic equivalence} should hold and found the proof of Proposition \ref{prop: generic equivalence} without further AI assistance. The authors take full responsibility for the correctness of the mathematical content of this article.
\subsection*{Notation}
We use the following notation.

	Arithmetic:
\begin{itemize}
	\item $k$ a field of positive characteristic $p$.
	\item $F/\mathbb{Q}_{p}$ a finite extension with ring of integers $\mcO_F$ and residue field $k_{F}$ with $q$ elements.
	Often we will abbreviate $\mcO=\mcO_F$.
	\item $\breve{F}$ the completion of the maximal unramified extension of $F$ inside some fixed algebraic closure $\overline{F}$.
\end{itemize}
	Cohomological:
	\begin{itemize}
		\item $\Lambda$ a commutative ring which is flat over $\mbZ$ and in
        which $p$ is invertible.
        \item For a finite-dimensional Noetherian scheme $Y$, $\DA(Y)$
        denotes the stable $\infty$-category
        $\DA_{\acute{e}t}(Y,\Lambda)$ of \'etale motives considered in
        \cite[Section 1.2]{preisUlaMotives}. The canonical functor
        \[
            \DA_{\acute{e}t}(Y,\Lambda)
            \longrightarrow \mathrm{DM}_{h}(Y,\Lambda)
        \]
        is a symmetric monoidal equivalence compatible with the six
        operations by \cite[Theorem 1.2.15]{preisUlaMotives}. Preis's
        comparison applies, in particular, to schemes finitely presented
        over a finite-dimensional Noetherian base. Over a perfect field, we
        also use the extension of $\mathrm{DM}_{h}$ from finite-type schemes
        to perfectly finitely presented schemes and perfect prestacks
        constructed in
        \cite[Proposition 2.2 and Section 2.2]{vandenHoveRamifiedMotivicSatake}
        and, using Preis's comparison on finite-type schemes, denote the
        resulting categories again by $\DA(Y)$.
        \item For a morphism $f\colon Y\rightarrow T$ of schemes finite type over a field, we write $\DA{\ula}(Y/T)$ for the category of étale motivic sheaves on $Y$ that are ULA with respect to $f\colon Y\rightarrow T$ in the sense of \cite[Definition 3.2.2]{preisUlaMotives}. In particular,
        \[
        \DA\ula(T/T)=\DA\rig(T)
        \]
        are the dualizable objects,
        see \cite[Remark 3.2.4]{preisUlaMotives}.
        For a perfectly finitely presented morphism $f\colon Y\rightarrow T$ of perfect schemes over a perfect field $k$ or a morphism of pre-stacks on perfect $k$-schemes, we will use the same symbols but use the six functor formalism discussed in the previous point to give sense to this.
        Finally, when $k$ is an algebraically closed field, we write $\Psi_f$ for the total motivic nearby cycles functor constructed by Ayoub and extended to the $\infty$-categorical set-up we work in by Preis in\cite[Section 1.4]{preisUlaMotives}; again, we denote by the same symbol the extension to pre-stacks on perfect $k$-schemes.
        For a perfect pre-stack $Y$ over a perfect field $k$, we denote by $\DA_{c}(Y)$ the $\infty$-category of constructible étale motives as constructed by van den Hove \cite[Appendix A.1]{vandenHoveSplittingModels}.
        It follows from \cite[Proposition 3.2.21]{preisUlaMotives} that $\DA_{c}(Y)=\DA{\ula}(Y/k)$.
\end{itemize}

\section{Recollections on the Bando affine Grassmannian}
Let $\Perf_{k_{F}}$ be the category of perfect $k_{F}$-schemes.
As introduced by Bando, let $\Divtil\colon \Perf_{k_{F}}^{\op}\rightarrow \Set$ be the étale sheaf whose value on $\Spec(R)\in \Perf_{k_{F}}$ is given by
$$
\Divtil(R)=\lbrace (\xi_{1},\xi_{2})\in W_{\mcO}(R)[\![t]\!]^{2}\colon \xi_{1}=[a]\pi+[b]t,\xi_{2}=[c]\pi+[d]t, ad-bc\in R^{*} \rbrace.
$$
Note that $\Divtil\simeq \GL_{2,k_{F}}^{\perf}$ via the obvious map.
For a perfect field $k$ that is an extension of $k_{F}$, we will write $\Divtilk$ for the restriction of $\Divtil$ to perfect $k$-algebras. 

Let $(\xi_{1},\xi_{2})\in \Divtil(R).$
Then $\xi_{1},\xi_{2}$ form a Koszul regular sequence in $W_{\mcO}(R)[\![t]\!]$ and we have $W_{\mcO}(R)[\![t]\!]/(\xi_{1},\xi_{2})\simeq R.$
Namely, by definition, we find for $\xi_{1},\xi_{2}\in \Divtil(R)$ an invertible matrix $A\in \GL_{2}(W_{\mcO}(R)[\![t]\!])$ such that $A\cdot (\pi,t)=(\xi_{1},\xi_{2})$ and then for the Koszul complexes we have
$$
\Kos(W_{\mcO}(R)[\![t]\!];\xi_{1},\xi_{2})\simeq \Kos(W_{\mcO}(R)[\![t]\!],\pi,t)
$$
(\cite[Tag 0625]{stacks}) and one checks that $(\pi,t)$ forms a regular sequence in $W_{\mcO}(R)[\![t]\!]$.
Let us write $B^{+}_{\xi_{1},\xi_{2}}(R):=\widehat{(W_{\mcO}(R)[\![t]\!]/(\xi_{1}))}_{(\xi_{2})}\simeq W_{\mcO}(R)[\![t]\!]/(\xi_{1})$ (\cite[Remark 2.3]{bandoComparison}) and $B_{\xi_{1},\xi_{2}}(R):=B^{+}_{\xi_{1},\xi_{2}}(R)[1/\xi_{2}].$
These are $\mcO[\![u]\!]$-algebras via $u\mapsto \xi_{2}.$
Although the rings $B^{+}_{\xi_{1},\xi_{2}}(R),$ $B_{\xi_{1},\xi_{2}}(R)$ resp. the $\mcO[\![u]\!]$-algebra structure depend on the choice of $(\xi_{1},\xi_{2})$ (resp. $\xi_{2}$), in the following we will just write $B^{+}(R)=B^{+}_{\xi_{1},\xi_{2}}(R)$ and $B(R)=B_{\xi_{1},\xi_{2}}(R)$ when no confusion can arise.

Let $\mcG$ be a fiberwise connected, smooth affine group scheme over $\mcO[\![u]\!].$
Then $L^{+}\mcG\rightarrow \Divtil$ and $L\mcG\rightarrow \Divtil$ are defined by $L^{+}\mcG(R)=\mcG(B^{+}(R))$ and $L\mcG(R)=\mcG(B(R)).$
For future use, we also collect the following statement.
\begin{Lemma}\label{Lemma: Quoient by congruence subgroup base changed from a field}
    Let $(L^{+}\mcG)^{\geq m}\subseteq L^{+}\mcG$ be the subgroup sheaf $(L^{+}\mcG)^{\geq m}=\ker(\mcG(B^{+}(R))\rightarrow \mcG(B^{+}(R)/\xi_{2}^{m})).$
    Then there are natural isomorphisms
    $$
    L^{+}\mcG/(L^{+}\mcG)^{\geq 1} \simeq \overline{\mcG} \times_{\Spec(k_{F})} \Divtil
    $$
    and
    $$
    (L^{+}\mcG)^{\geq m}/(L^{+}\mcG)^{\geq m+1}\simeq (\Lie(\overline{\mcG})\times_{\Spec(k_{F})} \Divtil)\lbrace m \rbrace,
    $$
    where $m\geq 1$, $\overline{\mcG}=\mcG\otimes_{\mcO[u]}k_{F}$ and we denote by $\lbrace m \rbrace$ the twist by $\xi_{2}^{m}B^{+}(R)/\xi_{2}^{m+1}B^{+}(R).$
    Here the map $\mcO[u]\rightarrow k_{F}$ is given by $u\mapsto 0$ and by taking the quotient by $(\pi_{F})$.
\end{Lemma}
The proof is identical to \cite[Lemma 2.7]{bandoComparison}.
We obtain Bando's affine Grassmannian $\Gr^{B}_{\mcG}=(L\mcG/L^{+}\mcG)_{\et}$ (étale sheafification). It is an étale sheaf on $\Perf_{k_{F}}$ over $\Divtil.$
The next statement is essentially contained in \cite{bandoComparison}.
\begin{proposition}[Bando]\label{Proposition: representability Bando affine Grassmannian}
    $\Gr^{B}_{\mcG}$ is representable by the increasing union along closed immersions of perfect schemes over $\Divtil$ that are perfectly of finite presentation and perfectly quasi-projective.
\end{proposition}
\begin{proof}
    In case $\mcG$ is split, this is due to Bando \cite[Theorem 3.26]{bandoComparison}.
    In general, one may use \cite[Corollary 10.7]{PappasZhuLocalModels} to find a closed immersion $\mcG\hookrightarrow \GL_{n},$ such that the flat quotient $\GL_{n}/\mcG$ is representable by a quasi-affine scheme.
    This implies the desired result arguing as in \cite[Proposition 1.2.6]{ZhuIntroAffGrass}.
\end{proof}
In the rest of this section, we continue to analyze this affine Grassmannian for Pappas--Zhu group schemes $\mcG$ over $\mcO[u]$ as constructed in \cite{PappasZhuLocalModels}.
Our aim is to produce as usual a stratification indexed by elements of an appropriate Iwahori--Weyl group.
\subsection{Recollections on the Pappas--Zhu group scheme}\label{Subsection: Recollections on PZ group schemes}
In this section, we recall the construction due to Pappas--Zhu of extensions of parahoric models of tamely ramified reductive groups over $F$ to smooth fiberwise connected affine group schemes $\mcG$ over $\mcO[u]$.
Furthermore, we add some complements needed in our situation.

We work in the following setting.
\begin{Situation}\label{Situation: tamely ramified group}
Let $G$ be a connected reductive $F$-group.
We assume that there exists a tamely ramified splitting field for $G.$
In our set-up and notation we mostly follow Pappas--Zhu and
a reference for the following facts and unexplained notation is \cite[Section 1.d.2]{PappasZhuLocalModels}.
For the convenience of the reader and to fix notation, let us spell out what we need here:
By Steinberg's theorem \cite[Chapitre III, 2.3]{SerreCohoGalois}, $G_{\breve{F}}$ is quasi-split. We may then make the following additional assumptions:
There exists a finite Galois extension $\widetilde{F}/F$ contained in our fixed algebraic closure $\overline{F}$ splitting $G$. Denote by $F\subseteq F_{0}\subseteq \widetilde{F}$ the maximal unramified subextension with ring of integers $\mcO\subseteq \mcO_{0}\subseteq \widetilde{\mcO}$. 
We require that
\begin{enumerate}
    \item[$\bullet$] $G_{F_{0}}$ is quasi-split,
    \item[$\bullet$] $\Gamma=\Gal(\widetilde{F}/F)=\langle \sigma \rangle \ltimes \langle \gamma_{0} \rangle,$ where $\sigma$ is the $q$-power Frobenius generator of $\Gal(F_{0}/F)$ and $I=\langle \gamma_{0} \rangle\simeq \mathbb{Z}/e$ is the inertia group and $(e,p)=1$ and we have the relation $\sigma \gamma_{0}\sigma^{-1}=\gamma_{0}^{q},$
    \item[$\bullet$] there exists a uniformizer $\widetilde{\pi}\in \widetilde{\mcO}$ with $\widetilde{\pi}^{e}=\pi.$ 
\end{enumerate}
We will also consider the cover $\mcO[u]\rightarrow \mcO_{0}[v]$ given by $u\mapsto v^{e}.$
The group $\Gamma$ acts on $\mcO_{0}[v]$ via $\sigma$ on the coefficients and $\gamma_{0}$ acts by $v^{i}\mapsto \zeta^{i} v^{i},$ where $\zeta=\gamma_{0}(\widetilde{\pi})/\widetilde{\pi}\in \mcO_{0}$ (primitive $e$-th root of unity).
Recall that a \emph{rigidification} of $G$ (\cite[Definition 1.7]{PappasZhuLocalModels}) is the datum of $(G,A,S,P)$ where $A\subset G$ is an auxiliary maximal $F$-split torus, $A\subset S \subset G$ is a maximally $\breve{F}$-split torus defined over $F$ and $P$ is a minimal parabolic, which then contains $M=\Cent_{G}(A)$. 
Let $T=\Cent_{G}(S).$ 
This is a maximal torus of $G$ defined over $F$ since $G_{\breve{F}}$ is quasi-split. 
It splits over $\widetilde{F}$.
We will write $M=\Cent_{G}(A)$.

Let us furthermore fix some notation concerning split resp. quasi-split forms of $G$.
Let $(H,T_{H},B_{H},\varepsilon)$ be a split pinned reductive group over $\mathbb{Z}$ and assume $G\otimes_{F}\widetilde{F}\simeq H\otimes_{\mathbb{Z}}\widetilde{F}$.
We adjust this isomorphism such that it identifies the maximal torus $T$ of $G$ with the maximal split torus $T_{H}\otimes_{\mathbb{Z}}\widetilde{F}$ given by the pinning.
In this situation, we obtain an action of $\Gamma$ on the based root-datum $\Sigma_{H}=(X^{*}(T_{H}),\Delta,X_{*}(T_{H}),\Delta^{\vee})$.

This gives rise to a pinned quasi-split form $(G^{*},T^{*},B^{*},\varepsilon^{*})$ over $F$ of $H$; and therefore also of $G$ over $F$.
Concretely, we have $G^{*}=(\Res^{\widetilde{F}}_{F}(H_{\widetilde{F}}))^{\Gamma}$.
The group $G^{*}$ is an inner twist of $G$ over $F$ given by a class in $H^{1}(\Gamma, G^{*}_{\ad}(\widetilde{F}))$.
As assumed above, $G_{F_{0}}$ is quasi-split, so that we can realize this inner twist of $G$ over $F_{0}$ already.
By \cite[Corollary 1.10]{PappasZhuLocalModels}, there exists a unique class $[c^{\rig}]\in H^{1}(\widehat{\mathbb{Z}},N^{\prime *}(\breve{F}))$ whose image in $H^{1}(F, G^{*}_{\ad})$ gives the inner twist from above.
Let us write $\Int(g)\in M^{\prime *}(\breve{F})$ for the value of this cocycle at $1$.
Here $M^{\prime *}$ denotes the Levi subgroup of $G^{*}_{\ad}$ that corresponds to $M^{*}\subset G^{*}$ and $N^{\prime *}=N_{M^{\prime *}}(T^{*}_{\ad})$.
We will choose an inner isomorphism $\psi\colon G_{\breve{F}}\simeq G^{*}_{\breve{F}}$ such that the image of $\Int(g)$ under $M^{\prime *}(\breve{F})\rightarrow M^{*}_{\ad}(\breve{F})$ is described as in \cite[Proposition 1.14]{PappasZhuLocalModels}.
Then we obtain the description $G(R)=G^{*}(\breve{F}\otimes_{F}R)^{\widehat{\mathbb{Z}}}$, where the action of the topological generator $1\in \widehat{\mathbb{Z}}$ is given by $\Int(g)\cdot \sigma$ and $R$ is an $F$-algebra.
\end{Situation}

Now let us discuss some preliminaries concerning Bruhat--Tits buildings.

Let $K$ be a discretely valued henselian valuation field with valuation ring $K^{\circ}$ and perfect residue field.
Whenever we consider a field like $K,$ let us implicitly also choose an identification $v(K^{*})=\mathbb{Z}$ (in other words; we will also choose a uniformizer).
Let $B(H;K)$ be the (non-extended) Bruhat--Tits building of $H$ over $K$.
We have the apartment $\mcA(H,T_{H};K)$.
Under our assumptions, for any two discretely valued henselian valuation fields with perfect residue fields $K,K^{\prime}$ we have canonical identifications
\begin{equation}\label{eq: identification apartments split group}
    \mcA(H,T_{H};K)=\mcA(H,T_{H};K^{\prime}).
\end{equation}
For $k\in \Perf_{k_{F}}$ a field, $(\xi_{1},\xi_{2})\in \Divtil(k),$ the ring $B(k)$ is a discretely valued complete valuation field with fixed uniformizer $\xi_{2}\in B(k)$ and valuation ring $B^{+}(k)$ (see \cite[Lemma 2.6]{bandoComparison}).
We may then consider bounded convex subsets $\Omega\subset \mcA(H,T_{H};B(k))$ and transfer them under (\ref{eq: identification apartments split group}) for different choices of $(\xi_{1},\xi_{2})$.

Let $(\underline{G}^{*},\underline{T}^{*},\underline{B}^{*},\underline{\varepsilon}^{*})$ be the pinned quasi-split group over $\mcO[u^{\pm}]$ constructed by Pappas--Zhu in \cite[Section 2.c.1]{PappasZhuLocalModels}.
Concretely, it is given by $\underline{G}^{*}=(\Res^{\mcO_{0}[v^{\pm}]}_{\mcO[u^{\pm}]}H_{\mcO_{0}[v^{\pm}]})^{\Gamma}$ with the diagonal action $\tau(\gamma)\otimes \gamma,$ where $\tau\colon \Gamma\rightarrow \Aut_{\mcO}(H)$ comes from the action on the index root datum for $G.$

Assume that $\underline{G}=\underline{G}^{*}$ is quasi-split; it is split over the extension $\mcO_{0}[v^{\pm}]$ of $\mcO[u^{\pm}].$
Let $\underline{S}^{*}\subseteq \underline{T}^{*}$ be a maximal $\breve{\mcO}[u^{\pm}]$-split subtorus and $\underline{A}^{*}\subseteq \underline{T}^{*}$ be a maximal $\mcO[u^{\pm}]$-split subtorus.
Then $\Cent_{\underline{G}^{*}}(\underline{S}^{*})=\underline{T}^{*}$.
Let $C\in \Perf_{k_{F}}$ be algebraically closed and take $(\xi_{1},\xi_{2})\in \Divtil(C)$.
Note for later use that $B^{+}(C)$ is then an $\breve{\mcO}$-algebra.
We may consider the maximal $B(C)$-split torus $\underline{S}^{*}_{B(C)}=\underline{S}^{*}\otimes_{\mcO[u]}B(C).$
For different choices $(\xi_{1},\xi_{2}),(\xi^{\prime}_{1},\xi^{\prime}_{2})\in \Divtil(C)$, we obtain from (\ref{eq: identification apartments split group}) identifications of apartments in the respective buildings
\begin{equation}\label{eq: identification apartments quasi split groups}
\mcA(\underline{G}^{*}_{B_{(\xi_{1},\xi_{2})}(C)},\underline{S}^{*}_{B_{(\xi_{1},\xi_{2})}(C)};B_{(\xi_{1},\xi_{2})}(C))=\mcA(\underline{G}^{*}_{B_{(\xi^{\prime}_{1},\xi^{\prime}_{2})}(C)},\underline{S}^{*}_{B_{(\xi^{\prime}_{1},\xi^{\prime}_{2})}(C)};B_{(\xi^{\prime}_{1},\xi^{\prime}_{2})}(C)).
\end{equation}
This follows from \cite[Theorem 3.17]{PrasadTameDescentBT}; see also \cite[Remarque 3.4.2(2)]{JoaoAffGrassZ} for a more general statement.
Therefore, as before, we may transfer under (\ref{eq: identification apartments quasi split groups}) convex bounded subsets $\Omega$ contained in the apartment $\mcA(\underline{G}^{*}_{B_{(\xi_{1},\xi_{2})}(C)},\underline{S}^{*}_{B_{(\xi_{1},\xi_{2})}(C)};B_{(\xi_{1},\xi_{2})}(C))$
among different choices of $(\xi_{1},\xi_{2})$.

Finally, let $(G,A,S,P)$ be a rigidification of $G$.
We then have the maximal torus $T=\Cent_{G}(S)$ in $G.$
This data extends to $(\underline{G},\underline{A},\underline{S},\underline{P},\underline{T})$ by \cite[Section 2.c.4]{PappasZhuLocalModels}.
Let us note here that what we denoted by $G^{\flat}$ in the introduction is given by $G^{\flat}=\underline{G}\otimes_{\mcO[u^{\pm}]}k_{F}[u^{\pm}]$.
By construction, $\underline{G}$ is an inner twist of $\underline{G}^{*}$ given by a class $[\underline{c}^{\rig}]\in H^{1}(\widehat{\mathbb{Z}},\underline{N}^{\prime *}(\breve{\mcO}[u^{\pm}]))$. 
As before, let us denote by $\Int(\bold{g})$ the element in $\underline{M}^{\prime *}(\breve{\mcO}[u^{\pm}])$ which is the value at $1$ of this cocycle.

For $(\xi_{1},\xi_{2}),(\xi^{\prime}_{1},\xi^{\prime}_{2})\in \Divtil(k_{F}),$ we obtain identifications induced by (\ref{eq: identification apartments quasi split groups})
\begin{align*}\label{eq: identification apartments groups}
    \mcA(\underline{G}_{B_{(\xi_{1},\xi_{2})}(k_{F})},\underline{A}_{B_{(\xi_{1},\xi_{2})}(k_{F})}; B_{(\xi_{1},\xi_{2})}(k_{F})) & = \mcA(\underline{G}_{B_{(\xi_{1},\xi_{2})}(\overline{k_{F}})},\underline{S}_{B_{(\xi_{1},\xi_{2})}(\overline{k_{F}})}; B_{(\xi_{1},\xi_{2})}(\overline{k_{F}}))^{\sigma} \\
     & = \mcA(\underline{G}^{*}_{B_{(\xi_{1},\xi_{2})}(\overline{k_{F}})},\underline{S}^{*}_{B_{(\xi_{1},\xi_{2})}(\overline{k_{F}})}; B_{(\xi_{1},\xi_{2})}(\overline{k_{F}}))^{\Int(\bold{g})\cdot \sigma} \\
     & = \mcA(\underline{G}^{*}_{B_{(\xi_{1}^{\prime},\xi_{2}^{\prime})}(\overline{k_{F}})},\underline{S}^{*}_{B_{(\xi_{1}^{\prime},\xi_{2}^{\prime})}(\overline{k_{F}})}; B_{(\xi_{1}^{\prime},\xi_{2}^{\prime})}(\overline{k_{F}}))^{\Int(\bold{g})\cdot \sigma} \\
     & = \mcA(\underline{G}_{B_{(\xi_{1}^{\prime},\xi_{2}^{\prime})}(\overline{k_{F}})},\underline{S}_{B_{(\xi_{1}^{\prime},\xi_{2}^{\prime})}(\overline{k_{F}})}; B_{(\xi_{1}^{\prime},\xi_{2}^{\prime})}(\overline{k_{F}}))^{\sigma}\\
     & \subset \mcB(\underline{G}_{B_{(\xi_{1}^{\prime},\xi_{2}^{\prime})}(k_{F})};B_{(\xi_{1}^{\prime},\xi_{2}^{\prime})}(k_{F})).
\end{align*}
\begin{Theorem}[Pappas--Zhu, \cite{PappasZhuLocalModels}]\label{Theorem: PZ}
    Let $(G,A,S,P)$ be as before a rigidification of $G$ and let $\Omega\subset\mcA(G,A;F)$ be a bounded convex subset.
    There exists a smooth, affine group scheme $\mcG_{\Omega}\rightarrow \Spec(\mcO[u])$ with connected fibers, such that  
    \begin{enumerate}
    \item[$(a)$] $\G_{\Omega}\otimes_{\mcO[u]}\mcO[u^{\pm}]=\underline{G},$
    \item[$(b)$] for any $(\xi_{1},\xi_{2})\in \Divtil(k_{F})$, $\mcG_{\Omega}\otimes_{\mcO[u]}B^{+}_{(\xi_{1},\xi_{2})}(k_{F})=\mcP_{\Omega},$ where $\mcP_{\Omega}$ is the Bruhat--Tits group scheme (\cite[Theorem 8.3.13]{KalethaPrasadBruhatTits}) over $B^{+}_{(\xi_{1},\xi_{2})}(k_{F})$ associated to the transfer of $\Omega$ to a subset of  $\mcB(\underline{G}_{B_{(\xi_{1},\xi_{2})}(k_{F})};B_{(\xi_{1},\xi_{2})}(k_{F}))$ as constructed above.
\end{enumerate}
\end{Theorem}
\begin{Remark}
    As we will see in the proof, a similar base change property also holds for bounded convex subsets $$\Omega \subset \mcA(\underline{G}_{B_{(\xi_{1}^{\circ},\xi_{2}^{\circ})}(C)},\underline{S}_{B_{(\xi_{1}^{\circ},\xi_{2}^{\circ})}(C)}; B_{(\xi_{1}^{\circ},\xi_{2}^{\circ})}(C)),$$
    where $C\in \Perf_{k_{F}}$ is algebraically closed and $(\xi_{1}^{\circ},\xi_{2}^{\circ})\in \Divtil(C)$. 
\end{Remark}
\begin{proof}
This is essentially contained in \cite[Theorem 3.1]{PappasZhuLocalModels} although they state the base-change compatibilities only for $(\pi,t)$ and $(t,\pi)$ and they only work with points in the Bruhat--Tits building.
Let us explain why it holds in the generality claimed in (b).
 
First consider the case when $G=H_{F}$ is split.   
The pinning $(B_{H},\varepsilon)$ of $(H,T_{H})$ gives a Chevalley valuation on the root datum $\Phi=\Phi(H,T_{H})$ and therefore a hyperspecial point $x_{0}\in \mcA(H,T_{H};K)$ where $K$ is any discretely valued henselian valuation field with perfect residue field and fixed choice of uniformizer as above. 
In \cite[Section 3.b.2]{PappasZhuLocalModels} Pappas--Zhu construct a uniquely determined smooth group scheme $\mcG_{x_{0},\Omega}$ over $\mcO[u]$.
To orient the reader, let us give a quick recollection of the construction and refer to \cite[Section 3.b.2]{PappasZhuLocalModels} for more details and references:
The hyperspecial point $x_{0}\in \mcA(H,T_{H};K)$ defines a valuation on the root groups $U_{a}(K)$ and we can then consider the filtration $\lbrace U_{a}(K)_{x_{0},r} \rbrace_{r\in R}$ of $U_{a}(K)$.
The convex subset $\Omega\subset \mcA(H,T_{H};K)$ defines a \emph{concave} function
$$
f(a):=f_{\Omega}(a):=\inf \lbrace \lambda\in R\colon a(x-x_{0})+\lambda\geq 0 \forall x\in \Omega \rbrace.
$$
Consider the subgroup $H_{x_{0},\Omega}(K)\subset H(K)$ generated by all $U_{a}(K)_{x_{0},f(a)}$ and by $T_{H}(\mcO_{K})$.
Then we have for the Bruhat--Tits group scheme $\mcP_{\Omega}$ over $\Spec(\mcO_{K})$ that
$$
\mcP_{\Omega}(\mcO_{K})=H(K)_{x_{0},\Omega}\subset H(K).
$$

Now consider the additive group scheme $u^{\lceil f(a) \rceil}U_{a}\otimes_{\mbZ}\mcO[u]=\Spec(\mcO[u,u^{-\lceil f(a) \rceil}]x)$ and consider $\mcT=T_{H}\otimes_{\mbZ}\mcO[u]$. These group schemes define a schematic root datum over $\mcO[u]$ and by birational glueing of the group law we obtain a smooth group scheme $\mcG_{x_{0},\Omega}$ over $\mcO[u]$ with a fiberwise dense open subscheme given by
$$
\mcV_{x_{0},\Omega}=\prod_{a\in \Phi^{-}}\mcU_{a,x_{0},\Omega}\times \mcT \prod_{a\in \Phi^{+}}\mcU_{a,x_{0},\Omega}.
$$
By \cite[Section 1.2.13,1.2.14]{BruhatTitsII} the group scheme $\mcG_{x_{0},\Omega}$ is uniquely determined by the schematic root datum.
Recall that the Bruhat--Tits group scheme $\mcP_{\Omega}$ over $\Spec(\mcO_{K})$ is similarly constructed via birational glueing of the group law from the base change of the schematic root datum we constructed above along $\mcO[u]\rightarrow \mcO_{K}$, $u\mapsto \pi$.
Since Bruhat--Tits work in the generality needed here, we may use the uniqueness assertion in \cite[Section 1.2.13,1.2.14]{BruhatTitsII}, to deduce that
$$
\mcG_{x_{0},\Omega}\otimes_{\mcO[u]}\mcO_{K}\simeq \mcP_{\Omega,K}.
$$
If we specialize this discussion to $K=B_{(\xi_{1},\xi_{2})}(k)$ where $k\in \Perf_{k_{F}}$ is some perfect field, we deduce property (b) claimed in the Theorem. 
Property (a) is ensured by construction and in \cite[Section 3.b.2]{PappasZhuLocalModels} it is in particular shown that $\mcG_{x_{0},\Omega}=:\mcG_{\Omega}$ is affine (which is the difficult part) with connected fibers. This concludes the explanation of properties we need here in the split case.

Now let us assume that $G=G^{*}$ is quasi-split.
In \cite[Section 3.b.2, II),b)]{PappasZhuLocalModels}, we find the construction of a smooth affine group scheme with connected fibers $\mcG_{\Omega}\rightarrow \Spec(\mcO[u]).$ 
It is constructed from the split case: if $\Omega\subset \mcA(\underline{G}_{F},\underline{A}_{F};F)=\mcA(H_{\widetilde{F}},T_{H_{\widetilde{F}}};\widetilde{F})^{\Gamma}\subset \mcA(H_{\widetilde{F}},T_{H_{\widetilde{F}}};\widetilde{F})$ is bounded convex, we obtain by the split case that was dealt with before an affine smooth group scheme with connected fibers $\mcH_{\Omega}\rightarrow \Spec(\mcO_{0}[v])$.
By naturality of the construction, this group scheme carries a $\Gamma$-action and using this one constructs in a series of steps $\mcG_{\Omega}$.
Namely, let us \emph{only in this proof} write $\mcG_{\Omega}^{\prime}=(\Res^{\mcO_{0}[v]}_{\mcO[u]}\mcH_{\Omega})^{\Gamma}$\footnote{Later we will write $\mcG^{\prime}=\mcG_{\mcF^{\prime}}$ for the Pappas--Zhu group scheme associated to a facet $\mcF^{\prime}$; we hope that this does not lead to confusions later in this article.} and let $\mcG_{\Omega}=(\mcG_{\Omega}^{\prime})^{\circ}$ be the ('absolute') identity component which is shown by Pappas--Zhu in the proof of \cite[Theorem 3.1]{PappasZhuLocalModels} to be fiberwise connected. 
Sometimes we will write $\mcG_{\Omega}^{\PZ}$ resp. $\mcG_{\Omega}^{\prime \PZ}$ for these group schemes to distinguish them from the Bruhat--Tits group schemes.

We claim that $\mcG_{\Omega}$ satisfies the base change property with respect to $\mcO[u]\rightarrow B^{+}(k_{F})$ (resp. with respect to $\mcO[u]\rightarrow B^{+}(C)$, where $C\in \Perf_{k_{F}}$ is algebraically closed) as stated in (b) above.
Observe that given $\Spec(k)\rightarrow \Divtil$, with $k$ an $\mcO_{0}/\pi$-algebra, we obtain a finite free degree $e$ morphism $B^{+}(k)\rightarrow \widetilde{B^{+}}(k),$ where $\widetilde{B^{+}}(k)=B^{+}(k)[X]/(X^{e}-\xi_{2})$ is an $\mcO_{0}$-algebra that is a complete discrete valuation ring which comes along with a chosen uniformizer $\widetilde{\xi}_{2}\in \widetilde{B^{+}}(k)$ with the property that $\widetilde{\xi}_{2}^{e}=\xi_{2}.$
One may define a $\Gamma=\langle \sigma \rangle \ltimes \langle \gamma_{0} \rangle$-action on $\widetilde{B^{+}}(k)$ by letting $\sigma$ act via its natural action on $B^{+}(k)$ and $\gamma_{0}(\widetilde{\xi}_{2}):=\zeta \widetilde{\xi}_{2}$, where $\zeta\in \mcO_{0}$ is the primitive $e$-th root of unity given by $\gamma_{0}(\widetilde{\pi})/\widetilde{\pi}$ cf. Situation \ref{Situation: tamely ramified group} above.
This action has the property that $\widetilde{B^{+}}(k)^{\gamma_{0}}=B^{+}(k)$ and for the pair of divisors $(\pi,t)$ resp. $(t,\pi)$ this specializes to the natural action.
Furthermore, there exists an $\mcO[u]$-algebra homomorphism $\mcO_{0}[v]\rightarrow \widetilde{B^{+}}(k)$ given by $v\mapsto \widetilde{\xi}_{2}.$
In particular, $\underline{G}_{\widetilde{B}(k)}$ is split and isomorphic to $H\otimes_{\mathbb{Z}}\widetilde{B}(k).$

Let us assume that $k=C$ is algebraically closed and take any $(\xi_{1},\xi_{2})\in \Divtil(C)$.
Write $B^{+}(C):=B^{+}_{(\xi_{1},\xi_{2})}(C)$ and $B(C):=B_{(\xi_{1},\xi_{2})}(C)$.
As explained above in (\ref{eq: identification apartments quasi split groups}), we may consider $\Omega\subset \mcA(\underline{G}_{F},\underline{A}_{F};F)$ also as a bounded convex subset $\Omega\subset \mcB(\underline{G}_{B(C)};B(C))$ inside the apartment $\mcA(\underline{G}_{B(C)},\underline{S}_{B(C)};B(C)).$
By Bruhat--Tits \cite[Theorem 8.3.13]{KalethaPrasadBruhatTits}, we have the smooth affine group scheme with connected fibers $\mcP_{\Omega}\rightarrow \Spec(B^{+}(C))$ such that $\mcP_{\Omega}(B^{+}(C))=\underline{G}(B(C))^{0}_{\Omega}$ (pointwise stabilizer of $\Omega$ under $G(B(C))^{0}$ which is the kernel of the Kottwitz homomorphism).
Our claim is that $\mcG_{\Omega}\otimes_{\mcO[u]}B^{+}(C)\simeq \mcP_{\Omega}.$

We now show the claim.
Let us write $L=B(C),\widetilde{L}=\widetilde{B}(C),$ $R=B^{+}(C),\widetilde{R}=\widetilde{B^{+}}(C).$
By \cite[Theorem 8.3.2]{KalethaPrasadBruhatTits}, there exists a smooth affine $R$-group scheme $\mcG^{\dagger}_{\Omega,R},$ whose $R$-valued points agree with $\underline{G}(L)^{\dagger}_{\Omega}$ which is the stabilizer of $\Omega$ in $\underline{G}(L)^{1}$ (see \cite[Notation 2.6.15]{KalethaPrasadBruhatTits} for this notation). Furthermore, by the same Theorem, the relative identity component of $\mcG^{\dagger}_{\Omega,R}$ agrees with the connected Bruhat--Tits group scheme $\mcP_{\Omega}$ over $R$ such that $\mcP_{\Omega}(R)=\underline{G}(L)^{\circ}_{\Omega}$.

By the split case we dealt with before and by \cite[Proposition 7.7.1]{KalethaPrasadBruhatTits}, we see that $\mcH^{\prime,\PZ}_{\Omega}(\widetilde{R})\subseteq \mcH^{\dagger}_{\Omega}(\widetilde{R})$ is of finite index.
Passing to invariants under $\gamma_{0},$ we obtain that likewise $\mcG^{\prime,\PZ}_{\Omega}(R)\subseteq \mcG^{\prime}(L)^{\dagger}_{\Omega}$ is of finite index, where $\mcG^{\prime}=(\Res^{\mcO_{0}[v]}_{\mcO[u]}H_{\mcO_{0}[v]})^{\Gamma}$.
Since $\mcG^{\prime,\PZ}_{\Omega,R}$ and $\mcG^{\dagger}_{\Omega,R}$ have the same generic fiber $\underline{G}_{B(C)},$ we obtain by \cite[Corollary 2.10.11]{KalethaPrasadBruhatTits} an $R$-group scheme homomorphism $\mcG^{\prime, \PZ}_{\Omega,R}\rightarrow \mcG^{\dagger}_{\Omega,R}$ and we have $\mcG_{\Omega}^{\PZ}\subseteq \mcG^{\prime, \PZ}_{\Omega,R}$ is a smooth affine $R$-subgroup scheme with connected fibers.
Note that $\mcG_{\Omega,R}^{\PZ}(R)=(\mcG_{\Omega,R}^{\prime,\PZ})^{\circ}(R)\subseteq \mcG_{\Omega}^{\prime, \PZ}(R)$ is of finite index: this follows because $R$ is strictly henselian and therefore the quotient $\mcG_{\Omega}^{\prime,\PZ}(R)/\mcG_{\Omega,R}^{\PZ}(R)$ injects into the set of connected components $\pi_{0}(\mcG^{\prime,\PZ}_{\Omega,R}),$ which is finite.
Therefore, we see that $\mcG^{\PZ}_{\Omega}\rightarrow \mcG^{\dagger}_{\Omega,R}$ is a morphism of smooth $R$-group schemes, where the source has connected fibers and which has the property that $\mcG^{\PZ}_{\Omega}(R)\subseteq \mcG^{\dagger}_{\Omega,R}(R)=\mcG^{\prime}(L)^{\dagger}_{\Omega}$ is an open subgroup of finite index. 
Since by assumption we work with an algebraically closed residue field $C,$ we may use \cite[Lemma A.4.26]{KalethaPrasadBruhatTits} which precisely implies that in this situation we have that $\mcG^{\PZ}_{\Omega}(R)=(\mcG^{\dagger}_{\Omega,R})^{\circ}(R)$ while we have by definition that $(\mcG^{\dagger}_{\Omega,R})^{\circ}(R)=\mcP_{\Omega}(R)$.
We may then use the fully faithfulness result from \cite[Corollary 2.10.11]{KalethaPrasadBruhatTits} to conclude that
$$
\mcG_{\Omega}^{\PZ}\otimes_{\mcO[u]}B^{+}(C)\simeq \mcP_{\Omega},
$$
as claimed in part (b).
Using \cite[Corollary 2.10.11]{KalethaPrasadBruhatTits} again, this implies the desired claim in the case when $C$ is algebraically closed and the case of $k_{F}$ is deduced from this by \cite[Section 1.7.6]{BruhatTitsII}.

Now let us assume that $G$ is no longer necessarily quasi-split. 
Again, by Steinberg's theorem, we know that $G_{\breve{F}}$ is quasi-split and the arguments given by Pappas--Zhu in \cite[Section 3.b.2, (II), (c)]{PappasZhuLocalModels} construct a smooth affine group scheme $\mcG_{\Omega}$ over $\mcO[u]$ which then over $\breve{\mcO}[u]$ satisfies the base change compatibilities with respect to maps $\breve{\mcO}[u]\rightarrow B^{+}(\overline{k}_{F}).$
Using \cite[Section 1.7.6]{BruhatTitsII} again, we deduce then the desired base change compatibility over $\mcO[u].$
\end{proof}

Recall that we have chosen a rigidification $(G,A,S,P)$ of $G$ as in \cite[Definition 1.7]{PappasZhuLocalModels} and recall further the maximal torus $T=\Cent_{G}(S)$ in $G.$
Recall further that this data extends to $(\underline{G},\underline{A},\underline{S},\underline{P},\underline{T})$.
By the previous construction, Theorem \ref{Theorem: PZ}, there exist smooth affine group schemes $\mcS \subset \mcT\subset \mcG$ over $\mcO[u]$ with the property that for all $\mcO[u]\rightarrow B^{+}(C),$ the base change $\mcT\otimes_{\mcO[u]}B^{+}(C)$ resp. $\mcS\otimes_{\mcO[u]}B^{+}(C)$ identifies with the connected Néron model $\mcT_{B^{+}(C)}$ resp. $\mcS_{B^{+}(C)}$ of $\underline{T}_{B(C)}$ resp. $\underline{S}_{B(C)}$. 
\begin{Lemma}\label{Lemma: Iwahori Weyl group independent of choice}
    Let $C\in \Perf_{k_{F}}$ be an algebraically closed field and $(\xi_{1},\xi_{2}),(\xi^{\prime}_{1},\xi^{\prime}_{2})\in \Divtil(C)$.
    Then there exists an identification
    $$
    \widetilde{W}_{(\xi_{1},\xi_{2})}(C)\simeq \widetilde{W}_{(\xi^{\prime}_{1},\xi^{\prime}_{2})}(C),
    $$
    where
    $\widetilde{W}_{(\xi_{1},\xi_{2})}(C)=N_{\underline{G}}(\underline{T})(B_{(\xi_{1},\xi_{2})}(C))/\mcT(B^{+}_{(\xi_{1},\xi_{2})}(C))$
    resp. similarly for $(\xi^{\prime}_{1},\xi^{\prime}_{2})$.
\end{Lemma}
\begin{proof}
    This is standard, see e.g. \cite[Proposition 3.4.1]{JoaoAffGrassZ} for a similar argument.
    Consider the group $N_{\underline{G}}(\underline{
    T})(\breve{\mcO}[u^{\pm}])/\mcT(\breve{\mcO}[u])$ and we obtain maps to $\widetilde{W}_{(\xi_{1},\xi_{2})}(C)$ resp. $\widetilde{W}_{(\xi^{\prime}_{1},\xi^{\prime}_{2})}(C)$ via sending $u\mapsto \xi_{2}$ resp. $u\mapsto \xi_{2}^{\prime}$.
    To verify that these are isomorphisms, we may by the snake lemma reduce to the analogous question concerning the maps $N_{\underline{G}}(\underline{T})(\breve{\mcO}[u^{\pm}])/\underline{T}(\breve{\mcO}[u^{\pm}])\rightarrow N_{\underline{G}}(\underline{T})(B_{(\xi_{1},\xi_{2})}(C))/\underline{T}(B_{(\xi_{1},\xi_{2})}(C))$ given by $u\mapsto \xi_{2}$ (resp. similarly for $u\mapsto \xi_{2}^{\prime}$) and $\underline{T}(\breve{\mcO}[u^{\pm}])/\mcT(\breve{\mcO}[u])\rightarrow \underline{T}(B_{(\xi_{1},\xi_{2})}(C))/\mcT(B^{+}_{(\xi_{1},\xi_{2})}(C))$ given again by $u\mapsto \xi_{2}$ (resp. similarly for $u\mapsto \xi_{2}^{\prime}$).
    
    Let us explain why the first map is an isomorphism. For the second map, one may argue as in \cite[Proposition 3.4.1]{JoaoAffGrassZ}.
    We claim that
    $$
    H^{1}(\langle \gamma_{0} \rangle,T_{H}(\breve{\mcO}[v^{\pm}]))=0.
    $$
    Recall here that we are considering the \emph{diagonal} action of $\gamma_{0}$ on $T_{H}(\breve{\mcO}[v^{\pm}])$ where it acts on $T_H$ via $\tau\colon \Gamma\rightarrow \Aut(H)$ coming from action on the index root datum of $G$ and where $\gamma_{0}$ acts naturally on  $\breve{\mcO}[v^{\pm}]$.
    In other words, this cohomology group classifies étale torsors on $\Spec(\breve{\mcO}[u^{\pm}])$ under the torus $(\Res^{\breve{\mcO}[v^{\pm}]}_{\breve{\mcO}[u^{\pm}]}T_{H})^{\gamma_{0}}$.
    
    Granting this vanishing,
    this implies that 
    \begin{align*}
        N_{\underline{G}}(\underline{T})(\breve{\mcO}[u^{\pm}])/\underline{T}(\breve{\mcO}[u^{\pm}]) & = (N_{\underline{H}}(T_{H})(\breve{\mcO}[v^{\pm}])/T_{H}(\breve{\mcO}[v^{\pm}]))^{\gamma_{0}} \\
                    & = W(H,T_{H})^{\gamma_{0}}.
    \end{align*}
    Here we used in the last equality that $\Pic(\breve{\mcO}[v^{\pm}])=0$.
    Therefore the maps are isomorphisms, as claimed. 
    
    Now let us explain the vanishing.
    Recall first that for any $(\xi_{1},\xi_{2})\in \Divtil(C)$, the ring $B_{\xi_{1},\xi_{2}}(C)$ is a complete discretely valued field with algebraically closed residue field, so that the first étale cohomology of $\Spec(B_{\xi_{1},\xi_{2}}(C))$ with coefficients in any reductive group vanishes.
    Therefore, $$H^{1}_{\et}(\Spec(B_{\xi_{1},\xi_{2}}(C)),(\Res^{\breve{\mcO}[v^{\pm}]}_{\breve{\mcO}[u^{\pm}]}T_{H})^{\gamma_{0}})=0.$$
    But by direct inspection, one sees that for any choice of $(\xi_{1},\xi_{2})\in \Divtil(C)$, there is an isomorphism
    $$
    H^{1}(\langle \gamma_{0} \rangle,T_{H}(\breve{\mcO}[v^{\pm}]))\simeq H^{1}(\langle \gamma_{0} \rangle, T_{H}(\widetilde{B}_{(\xi_{1},\xi_{2})}(C))),
    $$
    so that the vanishing holds, as desired.
\end{proof}
Let us write $\widetilde{W}$ for the constant value of $\widetilde{W}_{(\xi_{1},\xi_{2})}(C)$.
For later use, let us remark that there exists a set-theoretic map 
$$
\widetilde{W}\rightarrow N_{\underline{G}}(\underline{T})(\breve{\mcO}(\!(u)\!)),
$$
denoted by $w\mapsto \dot{w}$,
which after base-change along $\breve{\mcO}(\!(u)\!)\rightarrow B_{(\xi_{1},\xi_{2})}(C)$, $u\mapsto \xi_{2}$, gives a section to the natural map $N_{\underline{G}}(\underline{T})(B_{(\xi_{1},\xi_{2})}(C))\rightarrow \widetilde{W}$. 
The existence of such a map can be justified as in \cite[Proposition 3.3]{EndoscopyMetaplectic} or follows from the proof we have given for Lemma \ref{Lemma: Iwahori Weyl group independent of choice}.

For facets $\mcF, \mcF^{\prime} \subset \mcA(G,A;F),$ recall that there are associated subgroups of the Iwahori--Weyl group $\widetilde{W}_{\mcF^{\prime}},\widetilde{W}_{\mcF}\subseteq \widetilde{W}$.
Recall that these are defined as follows: if $\mcP_{\mcF}$ resp. $\mcP_{\mcF^{\prime}}$ are the corresponding Bruhat--Tits group schemes over $\mcO$, then 
$$
\widetilde{W}_{\mcF}=(N_{G}(T)(\breve{F})\cap \mcP_{\mcF}(\breve{\mcO}))/\mcT(\breve{\mcO})
$$
and similarly for $\widetilde{W}_{\mcF^{\prime}}$.
If we transfer $\mcF,\mcF^{\prime}$ under the identification induced by (\ref{eq: identification apartments quasi split groups}), then the analogue of Lemma \ref{Lemma: Iwahori Weyl group independent of choice} holds true for $\widetilde{W}_{\mcF}$ and $\widetilde{W}_{\mcF^{\prime}}$ respectively.
In the following, we will always assume that $\mcF$, $\mcF^{\prime}$ are contained in the closure of a common alcove.
Recall that in that case, the quasi-Coxeter structure on $\widetilde{W}$ induces a length function on the double quotient $\widetilde{W}_{\mcF^{\prime}} \backslash \widetilde{W} / \widetilde{W}_{\mcF}$. This length function in turn determines an order and this order is called the Bruhat-order (see \cite[Lemma 1.6 ff]{RicharzJAlgebra}). As the notation already suggests, this length function is independent of the choice of $(\xi_{1},\xi_{2})\in \Divtil(C)$ used to define the appearing groups. 

\subsection{Stratification of the Bando affine Grassmannian}
We keep the notation from Situation \ref{Situation: tamely ramified group} above. Recall that we write $k:=\overline{k}_{F}$.
Let $\mcF, \mcF^{\prime} \subset \mcA(G,A;F)$ be facets contained in the closure of a common alcove. 
By the previous \S \ref{Subsection: Recollections on PZ group schemes}, we have the Pappas--Zhu group schemes $\mcG_{\mcF}=:\mcG,\mcG_{\mcF^{\prime}}=:\mcG^{\prime}$ over $\mcO[u].$
Consider $\Gr^{B}_{\mcG}$ and the left action $L\underline{G}\times \Gr^{B}_{\mcG}\rightarrow \Gr^{B}_{\mcG}$.
This induces a left action of $L^{+}\mcG^{\prime}\subseteq L\underline{G}$ on $\Gr^{B}_{\mcG}$.
\begin{Lemma}\label{Lemma: Iwahori Weyl group and double orbits}
    Let $C\in \Perf_{k_{F}}$ be algebraically closed and fix $(\xi_{1},\xi_{2})\in \Divtil(C)$. 
    There exists a bijection
    $$
    \widetilde{W}_{\mcF^{\prime}}\backslash \widetilde{W} / \widetilde{W}_{\mcF}\longrightarrow L^{+}\mcG^{\prime}(C)\backslash L\underline{G}(C)/L^{+}\mcG(C)=L^{+}\mcG^{\prime}(C)\backslash \Gr^{B}_{\mcG}(C).
    $$
\end{Lemma}
\begin{proof}
    Given Theorem \ref{Theorem: PZ}, this follows from \cite[Appendix, Proposition 8]{PappasRapoportTwistedLoopGroups}( see also \cite[Chapter 5]{KalethaPrasadBruhatTits}).
\end{proof}
Given $w\in \widetilde{W}$, we have $\dot{w}\in N_{\underline{G}}(\underline{T})(\breve{\mcO}(\!(u)\!))$ and we obtain a section $\dot{w}\in L\underline{G}(\Divtilk)$.
\begin{Definition}[Schubert cells and Schubert varieties]
    Let $w\in \widetilde{W}$ and consider $\dot{w}\in N_{\underline{G}}(\underline{T})(\breve{\mcO}(\!(u)\!))$ and let $k_{w}/k_{F}$ be the reflex field.\footnote{By Lemma \ref{Lemma: Iwahori Weyl group independent of choice} we know that $\widetilde{W}$ identifies with the Iwahori--Weyl group defined over $F$, so that it receives an action of $\Gal(\breve{F}/F)$ and then the \emph{reflex field} of $w$ is the one defined by the stabilizer of $w$ in $\Gal(\breve{F}/F)$.}
    
    The Schubert cell $\Gr^{B,w}_{\mcG}:={}^{\mcF^{\prime}}{\!}\Gr^{B,w}_{\mcG}\subset \Gr^{B}_{\mcG}$ associated to $w$ is the étale descent to $k_{w}$ of the $L^{+}\mcG^{\prime}\restriction_{\Perf_{\overline{k}_{F}}}$-orbit of $\dot{w}$.
    The Schubert variety associated to $w$ is the corresponding Zariski closure $\Gr^{B,\leq w}_{\mcG}=\overline{\Gr^{B,w}_{\mcG}}\subset \Gr^{B}_{\mcG}$.
\end{Definition}
The \emph{Hecke stack} is defined by
$$
\Hk{\mcG,\mcG^{\prime}}:=[L^{+}\mcG^{\prime}\backslash \Gr_{\mcG}^{B}]
$$
Although the Hecke stack depends on the pair $(\mcG,\mcG^{\prime})$, let us write in the following for simplicity $\Hk{}=\Hk{\mcG,\mcG^{\prime}}$ with the groups understood from the context.
For $w\in \widetilde{W}_{\mcF^{\prime}}\backslash \widetilde{W} / \widetilde{W}_{\mcF}$ we obtain substacks
$$
\Hk{}^{w}\subseteq \Hk{}^{\leq w}\subset \Hk{}
$$
given by $\Hk{}^{w}=[L^{+}\mcG^{\prime}\backslash \Gr_{\mcG}^{B,w}]$ and $\Hk{}^{\leq w}=[L^{+}\mcG^{\prime}\backslash \Gr_{\mcG}^{B,\leq w}]$.
\begin{Lemma}\label{lemma: structure of Orbits in Bando affine Grassmannian}
    Let $w\in \widetilde{W}_{\mcF^{\prime}}\backslash \widetilde{W} / \widetilde{W}_{\mcF}$.
    \begin{enumerate}
        \item[(a)] We have that $\Hk{}^{w}\simeq [\Divtil\restriction_{\Perf_{k_{w}}}/ \Stab_{L^{+}\mcG^{\prime}}(w)]$, where $\Stab_{L^{+}\mcG^{\prime}}(w)\subseteq L^{+}\mcG^{\prime}$ is the closed subgroup stabilizing $\dot{w}\in L\underline{G}(\Divtil\restriction_{\Perf_{\overline{k}_{F}}}),$
        \item[(b)] $\Stab_{L^{+}\mcG^{\prime}}(w)$ is isomorphic to $L^{+}\mcG_{\Omega},$ where $\Omega=\overline{\mcF^{\prime}\cup w\cdot \mcF}$ (convex hull). Therefore $$\Gr^{B,w}_{\mcG}\simeq L^{+}\mcG^{\prime}/ \Stab_{L^{+}\mcG^{\prime}}(w)$$ is a perfectly smooth scheme, perfectly of finite presentation over $\Divtil \restriction_{\Perf_{k_{w}}},$
        \item[(c)] $\Hk{}^{\leq w}=\coprod_{v\leq w} \Hk{}^{v}$ as sets, where $'\leq '$ denotes the Bruhat-order. 
    \end{enumerate}
\end{Lemma}
\begin{proof}
    Item (a) follows by definition.
    Now we turn to item (b).
    Note that $\Stab_{L^{+}\mcG^{\prime}}(w)=L^{+}\mcG^{\prime}\cap \dot{w}L^{+}\mcG\dot{w}^{-1}$ and we have a closed immersion of perfectly pro-algebraic groups $L^{+}\mcG_{\Omega}\rightarrow \Stab_{L^{+}\mcG^{\prime}}(w)$ defined over $\Spec(k_{w}).$
    To show that this closed immersion is indeed an isomorphism, we may replace $k_{w}$ by $\overline{k}_{w}=\overline{k}_{F}$. In that case for $R\in \Perf_{k}$, the rings $B(R)$ resp. $B^{+}(R)$ are $\breve{\mcO}$-algebras and we may therefore assume that $G$ is quasi-split. 
    The case when $G$ is split is contained in \cite[Lemma 4.3.7]{RicharzScholbachIntersectionMotives}.
    Now let us go back to the notation in \S \ref{Subsection: Recollections on PZ group schemes}.
    By the split case, we have $\mcH_{\Omega}(\widetilde{B}^{+}(R))^{\gamma_{0}}=\Stab_{L^{+}\mcH^{\prime}}(w)(\widetilde{B}^{+}(R))^{\gamma_{0}}$ and passing to the relative identity component, which is a functorial procedure, we deduce the desired result. 
    Finally, item (c) follows from Lemma \ref{Lemma: Iwahori Weyl group and double orbits}; the closure relations are verified immediately.
\end{proof}

Finally, as a preparation for the constant term functor, we need to discuss connected components of $\Gr^{B}_{\mcG}$.
As usual, this reduces to a question about $L\underline{G}$.

Let $X_{*}(\underline{T})=\Hom_{\mcO_{0}[v^{\pm}]}(\mathbb{G}_{m,\mcO_{0}[v^{\pm}]},\underline{T}_{\mcO_{0}[v^{\pm}]})(=X_{*}(T_{H}))$.
This is a finite free $\mathbb{Z}$-module with an action by $I=\langle \gamma_{0} \rangle$.
Recall that we have the coroot lattice $\Phi^{\vee}_{\abs}=\Phi(H,T_{H})^{\vee}\subset X_{*}(\underline{T})$.
This lattice is stable under $\gamma_{0}$.
Consider the \emph{algebraic fundamental group} $\pi_{1}(\underline{G})=X_{*}(\underline{T})/\mathbb{Z}\cdot \Phi^{\vee}_{\abs}$.
Let us construct as usual a locally constant function
$$
\kappa_{\underline{G}}\colon |L\underline{G}|\rightarrow \pi_{1}(\underline{G})_{I}.
$$
We may argue as in \cite[Section 2.a.2]{PappasRapoportTwistedLoopGroups} to reduce the construction of the Kottwitz map for $C\in \Perf_{k_{F}}$ algebraically closed
$$
\kappa_{\underline{G}}\colon \underline{G}(B(C))\rightarrow \pi_{1}(\underline{G})_{I}
$$
to the case when $G=\mathbb{G}_{m}$, where it is simply given by the $\xi_{2}$-adic valuation.
This then defines the desired map $\kappa_{\underline{G}}$ on the underlying topological space $ |L\underline{G}|$ of $L\underline{G}$.
\begin{Lemma}\label{Lemma: connected components of Bando affine Grassmannian}
	The map $
	\kappa_{\underline{G}}\colon |L\underline{G}|\rightarrow \pi_{1}(\underline{G})_{I}
	$ is locally constant. 
\end{Lemma}
\begin{proof}
	This reduces to the case where $\underline{G}=\underline{T}$ with $T$ a torus or when $G$ is simply connected.
	In the simply connected case, we claim that $|L\underline{G}|$ is connected. But in this case, we may argue using affine root-groups to see that $L_{C}\underline{G}$, where $C\in \Perf_{k_{F}}$ is algebraically closed, is indeed connected (cf. \cite[Lemma 4.2]{AGLR}).
	In the case of a torus, we are furthermore reduced to $T=\mathbb{G}_{m}$ and then it finally follows from the fact that the $\xi_{2}$-adic valuation of an element $x\in
	B(R)^{*}$ is Zariski-locally constant on $|\Spec(R)|$.
\end{proof}
For $C\in \Perf_{k_{F}}$ an algebraically closed field, the Kottwitz map $\kappa_{\underline{G}}$ is trivial on $L^{+}\mcG(C)\subseteq L\underline{G}(C)$ by the base change compatibility in Theorem \ref{Theorem: PZ} and by \cite[Corollary 11.5.4]{KalethaPrasadBruhatTits}.
The following definition therefore makes sense.
\begin{Definition}\label{def: Kottwitz map on Bando aff Grass}
    The Kottwitz map on the Bando affine Grassmannian is
    $$
    \kappa_{\underline{G}}\colon |\Gr^{B}_{\mcG}|\longrightarrow \pi_{1}(\underline{G})_{I}
    $$
    induced by the map $\kappa_{\underline{G}}\colon |L\underline{G}|\rightarrow \pi_{1}(\underline{G})_{I}$.
\end{Definition}
Note that the Kottwitz map on the Bando affine Grassmannian is still locally constant.

We furthermore record the following property of the Bando affine Grassmannian which will be needed later on.
\begin{Lemma}\label{Lemma: Bando affine Grassmannian ind-projective for Pappas--Zhu group schemes}
    $\Gr^{B}_{\mcG}\rightarrow \Divtil$ over $\Perf_{k_{F}}$ is ind-perfectly proper.
\end{Lemma}
\begin{proof}
    We adapt the proof given by Pappas--Zhu \cite[Proposition 5.3]{PappasZhuLocalModels} to our situation.
    
    Let us first suppose that $\underline{G}=H\otimes_{\mathbb{Z}}\mcO[u^{\pm}]$ is split.
    Let $\mfa$ be an alcove in $\mcA(G,T;F)$ such that $\mcF\subseteq \overline{\mfa}$ (closure) and take $x\in \mcF$.
    If $y$ is in the interior of $\mfa$, then by the proof of \cite[Theorem 3.1]{PappasZhuLocalModels}, we obtain a smooth affine group scheme $\mcG_{y}$ associated to $y$ and a morphism $\mcG_{y}\rightarrow \mcG_{x}$ such that $\mcG_{y}\otimes_{\mcO[u]}\mcO[u^{\pm}]=\mcG_{x}\otimes_{\mcO[u]}\mcO[u^{\pm}]$.
    This implies that $\Gr^{B}_{\mcG_{y}}\rightarrow \Gr^{B}_{\mcG_{x}}$ is surjective and it suffices to show that $\Gr^{B}_{\mcG_{y}}$ is ind-perfectly proper.
    There exists a hyperspecial point $x_{0}\in \overline{\mfa}$.
    Then we have that $\mcG_{x_{0}}=H\otimes_{\mathbb{Z}}\mcO[u]$ and we have a morphism $\mcG_{y}\rightarrow \mcG_{x_{0}}$ realizing $\mcG_{y}$ as the dilatation of $H\otimes_{\mathbb{Z}}\mcO[u]$ along a Borel $B$ in the fiber $u=0$.
    By the functor of points description of the dilatation, and since the $B^{+}(R)$ are $\xi_{2}$-torsionfree and $B^{+}(R)/\xi_{2}\simeq R$, we obtain that the fpqc sheaf on perfect schemes over $\Divtil$ associated to $R\mapsto \mcG_{x_{0}}(B^{+}(R))/\mcG_{y}(B^{+}(R))$ is isomorphic to the perfectly projective $\Divtil$-scheme $(H_{k_{F}}/B_{k_{F}})^{\perf}\times_{\Spec(k_{F})}\Divtil$.
    Since it is known by Bando \cite[Theorem 3.26]{bandoComparison} that $\Gr^{B}_{\mcG_{x_{0}}}$ is ind-perfectly projective, we deduce the desired result for $\Gr_{\mcG_{y}}^{B}$.
    
    Now consider the general case.
    We may reduce to $\underline{G}$ is quasi-split and split over $\mcO[v^{\pm}]$ and $\mcO=\mcO_{0}$.
    Recall that $x\in \mcA(G,A;F)=\mcA(H_{\widetilde{F}},T_{H,\widetilde{F}};\widetilde{F})^{\gamma_{0}}$ and $\mcG_{x}$ is the relative identity component of $(\Res^{\mcO[v]}_{\mcO[u]}\mcH_{x})^{\gamma_{0}}$.
    By the last part of the proof of \cite[Theorem 3.1]{PappasZhuLocalModels}, we find a $\gamma_{0}$-invariant alcove $\mfa\subset \mcA(H_{\widetilde{F}},T_{H,\widetilde{F}};\widetilde{F})$ such that $x\in \overline{\mfa}$.
    Let $y\in \mfa$ be the barycenter which is then fixed by $\gamma_{0}$.
    Then consider $\mcH_{y}=:\mcH$ (Iwahori). Then the maximal reductive quotient $\overline{\mcH}^{\red}$ is a split torus over $\mcO$.
    Using Lemma \ref{Lemma: Quoient by congruence subgroup base changed from a field}, we see that the kernel of $\mcH(B^{+}(R))\rightarrow \overline{\mcH}^{\red}(R)$ is an affine perfectly pro-unipotent group scheme. 
    From here, the rest of the proof of \cite[Proposition 5.3]{PappasZhuLocalModels} adapts readily.
\end{proof}
\section{Hecke stack}\label{Section: Hecke stack}
We work over $k=\overline{k}_{F}$.
Let $G$ be a tamely ramified group over $F$ as in Situation \ref{Situation: tamely ramified group} but we additionally assume that $G=G^{*}$ is quasi-split since over $\breve{\mcO}[u^{\pm}]$ the reductive group scheme $\underline{G}$ constructed by Pappas--Zhu is quasi-split in any case. 
We keep the notation from Situation \ref{Situation: tamely ramified group} and in particular $e$ is the ramification index of the tamely ramified splitting field $\widetilde{F}$ of $G$. Let $(G,A,S,P)$ be a rigidification
of $G$. Fix a pair of facets $(\mcF,\mcF^{\prime})$ of $\mcA(G,A;F)$ contained in the closure of a common alcove, and let
$\mcG$ and $\mcG^\prime$ be the corresponding Pappas--Zhu parahoric group
schemes. Recall that the associated Hecke stack over $\Divtil$ is the quotient
\[
    \Hk{}=\left[L^{+}\mcG^{\prime}\backslash\Gr^{B}_{\mcG}\right].
\]
We study universally locally acyclic motives on $\Hk{}$. The definition and
the basic properties below parallel
\cite[Section VI.6]{farguesscholzeGeometrization}.

\subsection{Bounded and universally locally acyclic motives on the Hecke stack}\label{subsection: bounded and ula motives on Hecke stacks}

Let $T$ be a perfect $\Divtilk$-scheme and set
$\Hk{T}=\Hk{}\times_{\Divtilk}T$. We use the extension of $\DA(-)$ to perfect
prestacks recalled in the Notation section. Thus, $\DA(\Hk{T})$ is the category
of étale motives on the quotient prestack $\Hk{T}$. Denote the quotient map by
\[
    q_T:\Gr^{B}_{\mcG,T}\longrightarrow\Hk{T}.
\]

By Proposition \ref{Proposition: representability Bando affine Grassmannian}
and Lemma
\ref{Lemma: Bando affine Grassmannian ind-projective for Pappas--Zhu group schemes},
we may choose a filtered presentation
\[
    \Gr^{B}_{\mcG}=\underset{i\in I}\colim X_i
\]
by $L^+\mcG^\prime$-stable finite unions of Schubert varieties. Each $X_i$
is perfectly projective and perfectly of finite presentation over $\Divtilk$,
and the transition maps are closed immersions. For a perfect
$\Divtilk$-scheme $T$, put
\[
    X_{i,T}=X_i\times_{\Divtilk}T,\qquad
    \Hk{T,i}=[L^+\mcG^\prime_T\backslash X_{i,T}].
\]
For each $i$, the action on $X_i$ factors through
$L^{(m_i)}\mcG^\prime=L^+\mcG^\prime/(L^+\mcG^\prime)^{\geq m_i}$ for some
$m_i$. Lemma
\ref{Lemma: Quoient by congruence subgroup base changed from a field} and
homotopy invariance give
\[
    \DA(\Hk{T,i})\simeq
    \DA([L^{(m_i)}\mcG^\prime_T\backslash X_{i,T}]).
\]
The right-hand side is independent of the sufficiently large choice of
$m_i$.

\begin{Definition}\label{definition:boundedhecke}
	An object $M\in\DA(\Hk{T})$ is \emph{bounded} if it is supported on
	$\Hk{T,i}$ for some $i\in I$. Equivalently, $q_T^*M$ is supported on a
	finite union of Schubert cells. We denote the full subcategory of bounded
	motives by $\DA(\Hk{T})^{\bd}$.
\end{Definition}
\begin{Definition}\label{definition:ulahecke}
	We define $\DA\ula(\Hk{T})$ to be the full subcategory of
	$\DA(\Hk{T})^{\bd}$ consisting of those objects $M$ for which $q_T^*M$
	is universally locally acyclic over $T$.
\end{Definition}
\begin{Remark}
	Although Definition \ref{definition:ulahecke} appears to depend on the
	ordering of $(\mcF,\mcF^\prime)$, it is independent of this ordering by an
	argument similar to
	\cite[Proposition VI.6.2]{farguesscholzeGeometrization}.
\end{Remark}
For $w\in \widetilde{W}_{\mcF^{\prime}}\backslash \widetilde{W} /
\widetilde{W}_{\mcF}$, let
\begin{equation*}
	\Hk{}^{w} =\left[L^{+}\mcG^\prime\backslash\Gr^{B,w}_{\mcG}\right],\quad
	\Hk{}^{\leq w}=
	\left[L^{+}\mcG^{\prime}\backslash\Gr^{B,\leq w}_{\mcG}\right]
\end{equation*}
be the corresponding substacks of $\Hk{}$. For a perfect $\Divtilk$-scheme
$T$, let $\Hk{T}^w$ and $\Hk{T}^{\leq w}$ denote their base changes to $T$,
and define $\DA\ula(\Hk{T}^w)$ and $\DA\ula(\Hk{T}^{\leq w})$ as in
Definition \ref{definition:ulahecke}.

\begin{Lemma}\label{lemma:stalkwise}
	Let $T$ be a perfect $\Divtilk$-scheme. An object
	$$M\in \DA^{\bd}(\Hk{T})$$
	is universally locally acyclic over $T$ if and only if for all
	$w\in\widetilde{W}_{\mcF^{\prime}}\backslash
	\widetilde{W}/\widetilde{W}_{\mcF}$, the pullback
	$[w]^*M\in\DA(T)$ along the section $[w]:T\rightarrow \Hk{T}$ is
	dualizable.
\end{Lemma}
\begin{proof}
The proof of \cite[Proposition VI.6.5]{farguesscholzeGeometrization} applies
without change.
\end{proof}
\begin{Corollary}\label{corollary:closed}
Let $T$ be a perfect $\Divtilk$-scheme, and let $j$ be the locally closed
immersion of a Schubert stratum. The category $\DA\ula(\Hk{T})$ is stable
under Verdier duality, tensor products, internal homs, and
$j_!j^*,j_*j^*,j_!j^!,j_*j^!$. Moreover, these operations commute with
pullbacks in $T$.
\end{Corollary}
\begin{proof}
	By Lemma \ref{lemma:stalkwise}, the proof of
	\cite[Corollary VI.6.6]{farguesscholzeGeometrization} applies without
	change.
\end{proof}
The transition functors given by pushforward along the closed immersions
identify the bounded category with
\[
\DA(\Hk{T})^{\bd}\cong
\underset{i\in I}\colim\DA(\Hk{T,i}).
\]
Restricting to ULA motives gives a canonical equivalence
\begin{equation}\label{eq:filteredcolim}
\DA\ula(\Hk{T})\cong
\underset{i\in I}\colim\DA\ula(\Hk{T,i}).
\end{equation}
\subsection{Convolution}\label{subsection: Convolution}
For the remainder of this subsection, assume that $\mcG'=\mcG$. 
Our main Theorem \ref{Theorem: Main} claims our equivalence to be monoidal with respect to convolution on the full $\infty$-category. It is subtle to show that convolution really satisfies all the higher coherence needed to make it monoidal on the $\infty$-category; for étale sheaves on $p$-adic Hecke stacks there is a very nice and detailed explanation in \cite[Section 4.2]{padicCentralFunctor}.
In the context relevant for us, the necessary results have recently been obtained by Chowdhury in \cite[Section 2.3]{Chiro}.

Now let us introduce the convolution product in our setting.
For a
perfect $\Divtilk$-scheme $T$, let
\[
    \Hk{\mcG,T}^{(2)}
    =\left[L^+\mcG_T\backslash
      L\underline G_T\mathop{\times}^{L^+\mcG_T}
      \Gr^B_{\mcG,T}\right]
\]
be the convolution Hecke stack. Multiplication and the two projections give
the convolution correspondence
\[
    \Hk{\mcG,T}
    \xleftarrow{\ m_T\ }
    \Hk{\mcG,T}^{(2)}
    \xrightarrow{\ p_T\ }
    \Hk{\mcG,T}\times_T\Hk{\mcG,T}.
\]
For $M,N\in\DA\ula(\Hk{\mcG,T})$, put
\[
    M\star_TN=m_{T!}p_T^*(M\boxtimes_TN).
\]
On any pair of finite Schubert bounds, the morphism $m_T$ is perfectly
proper and its image is contained in a finite union of Schubert varieties.
Hence this formula defines a product on bounded motives.

Write
\[
    \mathcal A_T=\DA\rig(T),\qquad
    \mathcal C_T=\DA\ula(\Hk{\mcG,T}).
\]
Recall that pullback along the structure morphism
$h_T:\Hk{\mcG,T}\to T$ equips $\DA(\Hk{\mcG,T})$ with a
$\DA(T)$-module structure. This restricts to an action
\[
    \mathcal A_T\otimes\mathcal C_T\longrightarrow\mathcal C_T.
\]
On objects, it is given by $A\cdot M=h_T^*A\otimes M$.

\begin{Lemma}\label{lemma:convolutionaction}
    Convolution defines an exact $\mathcal A_T$-bilinear bifunctor
    \[
        \star_T:\mathcal C_T\times\mathcal C_T\longrightarrow\mathcal C_T.
    \]
    Equivalently, it induces an exact $\mathcal A_T$-linear functor
    \[
        \mathcal C_T\underset{\mathcal A_T}{\otimes}\mathcal C_T
        \longrightarrow\mathcal C_T.
    \]
    Furthermore, convolution defines a monoidal structure on $\mcC_{T}$.
\end{Lemma}
\begin{proof}
    Since ULA motives are dualizable for cohomological correspondences,
    relative external products preserve them by
    \cite[Definition 3.2.2]{preisUlaMotives}. On finite
    Schubert supports, all actions factor through
    \[
        L^{(N)}\mcG_T
        =L^+\mcG_T/(L^+\mcG_T)^{\geq N}
    \]
    for $N$ sufficiently large. On these supports, $p_T$ is represented by an
    $L^{(N)}\mcG_T$-torsor and is therefore smooth. Proper
    pushforward and smooth pullback preserve ULA motives by
    \cite[Lemma 3.2.8]{preisUlaMotives}. It follows that convolution restricts
    to $\mathcal C_T$.

    The relative external product is $\mathcal A_T$-bilinear, pullback along
    $p_T$ is $\mathcal A_T$-linear, and the projection formula equips
    $m_{T!}$ with an $\mathcal A_T$-linear structure. These are structures
    on functors of module categories, with their standard associativity and
    unit coherences. Their composite is therefore $\mathcal A_T$-bilinear.
    The second assertion follows from the universal property of the relative
    tensor product.
    Finally, the claim that convolution then defines a monoidal structure on $\mcC_{T}$ is explained in \cite[Section 2.3]{Chiro}.
\end{proof}
The following standard lemma follows from proper base change and compatibility of external products with pullback.
\begin{Lemma}\label{lemma:convolutionbasechange}
    Let $f:T'\to T$ be a morphism of perfect schemes over $\Divtilk$, and let
    $\widetilde f:\Hk{\mcG,T'}\to\Hk{\mcG,T}$ be the induced morphism. For
    $M,N\in\mathcal C_T$, there is a canonical equivalence
    \[
        \widetilde f^*(M\star_TN)
        \simeq
        \widetilde f^*M\star_{T'}\widetilde f^*N.
    \]
\end{Lemma}

\subsection{Generic equivalence}
We work with the pair $(\mcG^\prime,\mcG)$ and assume
$\mcG=\mcG^{\prime}$ only for convolution.
Let $X_e=D((1-x)^{2e}+x^{2e})\subseteq \mathbb{A}^{1,\perf}_{k}$.
Let us define the path
$$
h_{e}\colon X_{e}\rightarrow \GL_{2,k}^{\perf}\simeq \Divtilk
$$
given by
$$
x\longmapsto
    \begin{pmatrix}
        (1-x)^{e}&  -x^{e}\\
        x^{e}&(1-x)^{e}
    \end{pmatrix}.
$$
The localization performed to define $X_e$ ensures that this is a well-defined morphism.
In the following let us for simplicity just write $X:=X_e$ and $h:=h_e$.
Note that the points $0$ and $1$ in $\mathbb{A}^{1,\perf}_{k}$ actually lie in $X$ and that $h(0)=(\pi,t)$ and $h(1)=(-t,\pi)$.
Furthermore, note that the reflection isomorphism $r\colon \mbZ[x]\rightarrow \mbZ[x]$ given by $x\mapsto 1-x$ induces a reflection isomorphism $r\colon X\rightarrow X$.
This isomorphism has the property that $r(0)=1$.
Let $\eta\in X$ be the generic point of $\mathbb{A}^{1,\perf}_{k}$.
Let us write $X_{\eta}=\Spec(\mcO_{X_{\eta}})$ for the spectrum of the residue field at $\eta$.
Then we have that $\mcO_{X_{\eta}}=k(x)^{\perf}$, i.e. the perfection of the field of rational functions in $x$ over $k$.

Let $S_{a}$ be the perfection of the strict henselization of $D((1-x)^{2e}+x^{2e})\subset \mathbb{A}^{1}_{k}$ at $a\in \lbrace 0,1\rbrace$.
We have morphisms $S_{a}\rightarrow X$ and we denote by $S_{a,\eta}=S_{a}\times_{X}X_{\eta}$.
With this notation, the following isomorphism holds
$$
S_{0,\eta}\simeq S_{1,\eta}\times_{X_{\eta,r}}X_{\eta}.
$$
\begin{proposition}[Generic equivalence]\label{prop: generic equivalence}
    There is an isomorphism
    \[
        \rho_\eta:\Hk{S_{0,\eta}}\isomto\Hk{S_{1,\eta}}.
    \]
    Moreover, the induced equivalence
    $$ (\rho_\eta^{-1})^*:\DA\ula(\Hk{S_{0,\eta}})
    \isomto\DA\ula(\Hk{S_{1,\eta}}) $$
    is monoidal for convolution.
\end{proposition}
Let $\Hk{X_{\eta}}^{r}=\Hk{X_{\eta}}\times_{X_{\eta},r}X_{\eta}$. Proposition \ref{prop: generic
equivalence} follows from the following statement:
\begin{Lemma}\label{lem: key geometric statement about switch Hecke}
    There is an isomorphism
    $$
    \Hk{X_{\eta}}^{r}\simeq \Hk{X_{\eta}}
    $$
    over $X_{\eta}$. Moreover, this isomorphism is induced by identifications between the
    respective (positive) loop groups.
\end{Lemma}
\begin{proof}[Proof of Proposition \ref{prop: generic equivalence}]
    Recall that the diagram
    \[
    \begin{tikzcd}
        S_{0,\eta} \arrow[r] \arrow[d]
            & S_{1,\eta} \arrow[d] \\
        X_\eta \arrow[r,"r"'] & X_\eta
    \end{tikzcd}
    \]
    is Cartesian. Pulling back the isomorphism of Lemma
    \ref{lem: key geometric statement about switch Hecke} along
    $S_{0,\eta}\rightarrow X_\eta$ and using this diagram gives
    \begin{equation}\label{eq:generic-stack-isomorphism}
        \rho_\eta:\Hk{S_{0,\eta}}
        \isomto
        \Hk{S_{1,\eta}}.
    \end{equation}
    Thus $(\rho_\eta^{-1})^*$ is an equivalence. The loop-group
    isomorphisms defining $\rho_\eta$ identify the convolution
    correspondences, so this equivalence is monoidal.
\end{proof}
It remains to prove Lemma
\ref{lem: key geometric statement about switch Hecke}. We first compare
the Bando rings.
\begin{Lemma}\label{lem: comparing Bando rings in the generic equivalence}
    Let $R$ be a perfect $X_{\eta}$-algebra, $(\xi_{1}(x),\xi_{2}(x))=h_{e}(x)\in \Divtilk(R)$ and consider also $(\xi_{1}(1-x),\xi_{2}(1-x))=h_{e}(r(x))\in \Divtilk(R)$.
    Then the following statements hold
    \begin{enumerate}
        \item[(a)] $B^{+}_{\xi_{1}(x),\xi_{2}(x)}(R)\simeq W_{\mcO}(R)$ and  $B^{+}_{\xi_{1}(1-x),\xi_{2}(1-x)}(R)\simeq W_{\mcO}(R)$, functorially in $X_{\eta}$-algebra homomorphisms $R\rightarrow R^{\prime}$,
        \item[(b)] there exists a unit $\lambda\in W_{\mcO}(R)$ of the form $\lambda=\nu^{-e}$, such that  under the identification in (a), $\overline{\xi}_{2}(x)\in B^{+}_{\xi_{1}(x),\xi_{2}(x)}(R)$ and $\overline{\xi}_{2}(1-x)\in B^{+}_{\xi_{1}(1-x),\xi_{2}(1-x)}(R)$ differ by $\lambda$, i.e.
        $$
        \overline{\xi}_{2}(x)=\lambda \overline{\xi}_{2}(1-x).
        $$
        This unit is also functorial for $X_{\eta}$-algebra homomorphisms $R\rightarrow R^{\prime}$.
        \item[(c)] There is a $\gamma_{0}$-equivariant isomorphism of $B^{+}_{\xi_{1}(x),\xi_{2}(x)}(R)\simeq B^{+}_{\xi_{1}(1-x),\xi_{2}(1-x)}(R)$-algebras
        $$
        \widetilde{B^{+}}_{\xi_{1}(x),\xi_{2}(x)}(R)\simeq \widetilde{B^{+}}_{\xi_{1}(1-x),\xi_{2}(1-x)}(R)
        $$
        in the notation of the proof of Theorem \ref{Theorem: PZ}.
        This isomorphism is functorial in $X_{\eta}$-homomorphisms $R\rightarrow R^{\prime}$.
    \end{enumerate}
\end{Lemma}
\begin{proof}
    First note that our assumption explicitly means that
    $$
    \xi_{1}(x)=([1-x])^{e}\pi - [x]^{e}t
    $$
    and
    $$
    \xi_{2}(x)=[x]^{e}\pi+([1-x])^{e}t,
    $$
    while
    $$
    \xi_{1}(1-x)=[x]^{e}\pi-[1-x]^{e}t
    $$
    and
    $$
    \xi_{2}(1-x)=[1-x]^{e}\pi+[x]^{e}t.
    $$
    Here we write $x\in R$ for the image of $x$ under the structure morphism $k(x)^{\perf}\rightarrow R$.
    Therefore, $[1-x], [x]\in W_{\mcO}(R)[\![t]\!]$ are units, since they are already units in $W_{\mcO}(R)$ since this is true modulo $\pi$ and since $W_{\mcO}(R)$ $\pi$-adically separated and complete.
    Observe that we can write both $\xi_{1}(x)$ and $\xi_{1}(1-x)$ in the form $v(t-c)$, where $v\in W_{\mcO}(R)^{*}$ and $c=w\cdot \pi$ and $w\in W_{\mcO}(R)^{*}$.
    By Lemma \ref{lem: quotient by power series over pi adic ring} below, we obtain the isomorphisms in (a).
    Note that the isomorphism constructed from Lemma \ref{lem: quotient by power series over pi adic ring} does indeed satisfy the desired functoriality as claimed in (a).
    Now we turn to statement (b).
    Using the relations enforced in $B^{+}_{\xi_{1}(x),\xi_{2}(x)}(R)$ resp. in $B^{+}_{\xi_{1}(1-x),\xi_{2}(1-x)}(R)$, we obtain that
    $$
    \overline{\xi}_{2}(x)=(\frac{([1-x])^{2e}+[x]^{2e}}{([x])^{e}})\cdot \pi
    $$
    and
    $$
    \overline{\xi}_{2}(1-x)=(\frac{[x]^{2e}+([1-x])^{2e}}{([1-x])^{e}})\cdot \pi.
    $$
    The ratio of these two scalars in front of $\pi$ is given by
    $$
    \frac{([1-x])^{e}}{[x]^{e}}=\lambda.
    $$
    This satisfies the properties claimed in (b).
    Let us write $\nu=\frac{[x]}{([1-x])}$.
    Now we turn to statement (c).
    Recall first that
    $\widetilde{B^{+}}_{\xi_{1}(x),\xi_{2}(x)}(R)=B^{+}_{\xi_{1}(x),\xi_{2}(x)}(R)[U]/(U^{e}-\xi_{2}(x))$ and $\widetilde{B^{+}}_{\xi_{1}(1-x),\xi_{2}(1-x)}(R)=B^{+}_{\xi_{1}(1-x),\xi_{2}(1-x)}(R)[U]/(U^{e}-\xi_{2}(1-x))$ with $\gamma_{0}$ action given in both cases by $\gamma_{0}(U)=\zeta\cdot U$.
    Then we can just multiply $U$ by the scalar $\nu^{-1}$ from above to define the desired isomorphism.
\end{proof}
\begin{Lemma}\label{lem: quotient by power series over pi adic ring}
    Let $A$ be an $a$-adically complete and separated ring and $c\in (a)$.
    Then the $A$-algebra homomorphism $\phi\colon A[\![t]\!]\rightarrow A$ defined by $t\mapsto c$ is surjective with kernel given by $(t-c)$.
\end{Lemma}
\begin{proof}
    Note first that $\phi$ is well-defined by the universal property of the power series ring since $A$ is $a$-adically separated and complete.
    It is surjective since it is an $A$-algebra homomorphism.
    We need to see that claim about the kernel.
    For this, we search for $g(t)\in A[\![t]\!]$ such that the following equation holds
    \begin{equation}
        f-\phi(f)=(t-c)g(t).
    \end{equation}
    Solving this modulo higher powers of $t$, we obtain that we have to take
    $$
    g(t)=\sum_{k\geq 0}b_{k}t^{k},
    $$
    where
    $$
    b_{k}=\sum_{n>k}a_{n}c^{n-(k+1)}\in A
    $$
    is a well-defined expression because again $A$ is $a$-adically complete and separated.
    Now if $f\in \ker(\phi),$ then $\phi(f)=0$ and therefore $f\in (t-c)A[\![t]\!]$, as desired.
\end{proof}
\begin{Lemma}\label{lem: comparing loop groups generic equivalence}
    Consider $L^{a}_{X_{\eta}}\mcG=L^{a}\mcG\times_{\Divtilk,h_{e}}X_{\eta}$ and $L^{a,r}_{X_{\eta}}\mcG=L^{a}_{X_{\eta}}\mcG\times_{X_{\eta},r}X_{\eta}$ with $a\in \lbrace \emptyset, +\rbrace$.
    The following statements hold:
    \begin{enumerate}
        \item[(a)] There exists a canonical isomorphism of positive loop group sheaves on $\Perf_{X_{\eta}}$
        $$
        L^{+,r}_{X_{\eta}}\mcG\simeq L^{+}_{X_{\eta}}\mcG,
        $$
        \item[(b)] there exists a canonical isomorphism of loop group sheaves on $\Perf_{X_{\eta}}$
        $$
        L^{r}_{X_{\eta}}\mcG\simeq L_{X_{\eta}}\mcG,
        $$
    \end{enumerate}
\end{Lemma}
\begin{proof}
    Note that (b) follows from (a) together with Lemma \ref{lem: comparing Bando rings in the generic equivalence} (b).
    For (a), recall that we are assuming that $G$ is quasi-split.
    Therefore, as we have recalled in the proof of Theorem \ref{Theorem: PZ}, the Pappas--Zhu group scheme $\mcG$ is constructed from the split case as follows: let $\mcH$ be the Pappas--Zhu group scheme constructed from the split form $H$ of $G$ over $\widetilde{F}$, then $\mcG$ is the group scheme of the relative (=absolute here) identity component of $\Res^{\breve{\mcO}[v]}_{\breve{\mcO}[u]}(\mcH)^{\gamma_{0}}$.
    Therefore, the result follows from Lemma \ref{lem: comparing Bando rings in the generic equivalence} (c) together with the observation that the split Pappas--Zhu group scheme $\mcH$ over $\mcO[v]$ is equivariant as a subgroup scheme of $H\otimes_{\mbZ}\mcO[v^{\pm}]$ with respect to scaling $v$ by a unit by the uniqueness assertion in \cite[Section 1.2.13,1.2.14]{BruhatTitsII}.
\end{proof}
\begin{proof}[Proof of Lemma \ref{lem: key geometric statement about switch Hecke}]
    Apply the previous Lemma
    \ref{lem: comparing loop groups generic equivalence} to both
    $\mcG$ and $\mcG^\prime$. The resulting isomorphisms identify the loop
    group, the positive loop group action, and hence their quotient Hecke
    stacks.
\end{proof}
\subsection{Comparison of the special fibers by nearby cycles}
\label{subsection:nearby-comparison}
We use Ayoub's theory of motivic nearby cycles in the form studied in
\cite{preisUlaMotives}. We briefly extend it to the bounded quotient
stacks occurring below. On a finite Schubert bound, choose \(N\) such
that the positive loop group action factors through
\(L^{(N)}\mcG^\prime\), and apply \(\Psi\) termwise to finite-type
deperfections of the smooth \v{C}ech nerve of the resulting quotient
stack. Since \(\Psi\) is a specialization system
\cite[Remark 1.4.12(2)]{preisUlaMotives}, smooth base change makes these
functors compatible with the descent datum and therefore defines a functor
on the limit. Universal-homeomorphism invariance, together with
Lemma
\ref{Lemma: Quoient by congruence subgroup base changed from a field}
and homotopy invariance, makes the construction independent of the
deperfections and of \(N\). Proper base change makes it compatible with
enlargement of the Schubert bound.

By \cite[Definition 3.2.2]{preisUlaMotives}, ULA objects are precisely
the dualizable objects in the bicategory of cohomological
correspondences. By
\cite[Remark 3.3.15(2)]{preisUlaMotives}, the K\"unneth formula for
 nearby cycles makes the induced functor on cohomological
correspondences symmetric monoidal. Hence
\cite[Corollary A.8]{preisUlaMotives} shows that nearby cycles preserve
ULA objects. We thus obtain a functor
\[
    \Psi_f:
    \DA\ula(\Hk{S_a,\eta})
    \longrightarrow
    \DA\ula(\Hk{S_a,\sigma}).
\]
Recall that if \(E\in \DA\ula(\Hk{S_a})\) is universally locally acyclic over \(S_a\), then
\begin{equation}\label{eq:ula-nearby-restriction}
    \Psi_f(j_a^*E)\simeq i_a^*E
\end{equation}
by \cite[Proposition 3.3.17]{preisUlaMotives}.

Write
\[
    n_a:S_a\longrightarrow X,
    \qquad
    j_a:S_{a,\eta}\hookrightarrow S_a,
    \qquad
    i_a:\Spec(k)\hookrightarrow S_a
\]
for the canonical morphism from the henselization, inclusions of the generic and special points
respectively.
Let
\[
    h_a:\Hk{S_a}\longrightarrow S_a
\]
be the structure morphism. For a finite union $Y_X\subset\Hk{X}$
of closed Schubert strata, we write $\Psi_{Y,a}$ for the nearby cycles functor
associated with $Y_{S_a}\to S_a$. If $Y_X=\Hk{X}^{\leq w}$, we abbreviate this
to $\Psi_{w,a}$.

For a finite union \(Y_X\subset\Hk{X}\) of closed Schubert strata, write
\[
    Y_{S_a}=Y_X\times_XS_a,
    \qquad
    Y_a=Y_{S_a}\times_{S_a}k.
\]
Thus \(Y_0\) is the equal-characteristic special fiber and \(Y_1\) is the
Witt-vector special fiber. The isomorphism
\eqref{eq:generic-stack-isomorphism} respects the Schubert stratifications.
Let \(\rho_{\eta,Y}:Y_{S_0,\eta}\isomto Y_{S_1,\eta}\) be its
restriction. Then
\[
    (\rho_{\eta,Y}^{-1})^*:
    \DA(Y_{S_0,\eta})\isomto\DA(Y_{S_1,\eta}).
\]
The statement of the following proposition needs the next construction:
As we will see in the proof below, by homotopy-invariance of étale motives and Lemma \ref{Lemma: Quoient by congruence subgroup base changed from a field}, we have an equivalence
$$
\DA(\Hk{X}^{w})\simeq \DA([X/(H_{w}\times_{k}X)]),
$$
where $H_{w}=\overline{\mcG}_{w}$ in the notation of Lemma \ref{Lemma: Quoient by congruence subgroup base changed from a field} and $\mcG_{w}=\mcG_{\Omega_{w}}$ is the Pappas--Zhu group scheme associated to $\Omega_w=\overline{\mcF^\prime\cup w\mcF}$.
We therefore get via pull-back a functor
$$
\DA([\Spec(k)/(H_{w}\times_{k}\Spec(k))])\rightarrow \DA([X/(H_{w}\times_{k}X)])
$$
and therefore a functor
$$
\DA\cons(\Hk{k}^{\mathrm{eq},w})\rightarrow \DA{\ula}(\Hk{X}^{w}),
$$
which we denote by $Q\mapsto \underline{Q}_{X}$ (the 'constant extension to $X$').
\begin{proposition}\label{proposition:comparison-single-stratum}
For every
\(w\in\widetilde{W}_{\mcF^{\prime}}\backslash
\widetilde{W}/\widetilde{W}_{\mcF}\), there is an equivalence
\[
    \theta_w:
    \DA\cons(\Hk{k}^{\mathrm{eq},w})
    \isomto
    \DA\cons(\Hk{k}^{\mathrm{Witt},w}).
\]
It is defined by the following formula: for \(Q\in\DA_c(\Hk{k}^{\mathrm{eq},w}) \),
\begin{equation}\label{eq:stratum-specialization}
    \theta_w(Q)
    \simeq
    \Psi_{w,1}\bigl(
        (\rho_{\eta}^{-1})^{*}j_{0}^{*}n_{0}^{*}\underline{Q}_{X}
    \bigr).
\end{equation}
\end{proposition}

\begin{proof}
Let
\[
	\Omega_w=\overline{\mcF^\prime\cup w\mcF},
    \qquad
    P_{w,T}=\Stab_{L^+\mcG^\prime_T}(w).
\]
By Lemma \ref{lemma: structure of Orbits in Bando affine Grassmannian},
we have
\begin{equation}\label{eq:stratum-classifying-stack}
    \Hk{T}^w\simeq[T/P_{w,T}],
    \qquad
    P_{w,T}\simeq L^+\mcG_{\Omega_w,T}.
\end{equation}
Let $H_w=\overline{\mcG}_{\Omega_w}$ be as in Lemma \ref{Lemma: Quoient by congruence subgroup base changed from a field}. The quotient
$P_{w,T}/P_{w,T}^{\geq1}$ is isomorphic to $H_w\times_k T$, and the kernel is pro-unipotent. Thus,
homotopy invariance identifies the category of motives on
\(\Hk{T}^w\) with the category of motives on the
constant classifying stack $[T/(H_w\times_k T)]$; see
\cite[Proposition 2.2.1]{RicharzScholbachIntersectionMotives}.

Now consider $L^{+}_{X}\mcG_{\Omega_w}$, where we are considering the map $h\colon X\rightarrow  \Divtilk$ from the beginning of the previous subsection.
Similarly, consider again $L^{+,r}_{X}\mcG_{\Omega_w}=L^{+}_{X}\mcG_{\Omega_w}\times_{X,r}X$.
Let $R\in \Perf_{k}$, $\mcH$ be the Pappas--Zhu group scheme for the split form $H$ of $G$.
Then the $R$-valued points of the quotient of $L^{+}_{X}\mcG_{\Omega_w}$ resp. $L^{+,r}_{X}\mcG_{\Omega_w}$ by the first congruence subgroup are both given by the same formula:
$$
(\mcH(R[U]/(U^{e})))^{\gamma_{0},\circ}.
$$
However, the isomorphism in Lemma \ref{lem: comparing loop groups generic equivalence} over the generic fiber $X_{\eta}$ is given by multiplying $U$ by the scalar $\nu^{-1}$.
Therefore, taking the identity between the two quotients by the first congruence subgroups would not be compatible with the generic equivalence and the formula (\ref{eq:stratum-specialization}).
However, to show that the formula (\ref{eq:stratum-specialization}) does indeed give an equivalence as claimed we may argue as follows: 
Consider the group scheme whose $R$-valued points for an $k$-algebra $R$ are given by 
$$
K(R)=\Ker(\mcH(R[U]/(U^{e}))\rightarrow \mcH(R)).
$$
Note that
$$
\Ker((\mcH(R[U]/(U^{e})))^{\gamma_{0},\circ}\twoheadrightarrow \mcH(R))\subseteq K(R)
$$
is simply the intersection of $K(R)$ with $(\mcH(R[U]/(U^{e})))^{\gamma_{0},\circ}$ and therefore given by
$
K(R)^{\gamma_{0},\circ}.
$
Lemma \ref{Lemma: Quoient by congruence subgroup base changed from a field} has an obvious analogue for $\Ker(\mcH(R[U]/(U^{e}))\rightarrow \mcH(R))$ and we obtain that $K$ defines a normal, smooth, connected and split unipotent subgroup of the group scheme defined by $R\mapsto \mcH(R[U]/(U^{e}))$. 
Now, since $\gamma_{0}$ is of order prime to $p$, we see that $K(R)^{\gamma_{0}}$ is still smooth (\cite[3.4]{EdixhovenTame}) and since a subgroup of a unipotent group over a perfect field is still unipotent, we deduce that $
K(R)^{\gamma_{0},\circ}
$ defines likewise a normal, smooth, connected and split unipotent group.
Therefore, we may use homotopy invariance of étale motives and the fact that twist from the generic equivalence disappears once we mod out $U$, to see that the formula (\ref{eq:stratum-specialization}) does indeed define an equivalence, as desired.
\end{proof}
We finally come to our main result.
\begin{Theorem}\label{Theorem: Main}
There is an equivalence
\[
    \mcB_{\mcG',\mcG}:
    \DA\cons(\Hk{k}^{\mathrm{eq}})
    \isomto
    \DA\cons(\Hk{k}^{\mathrm{Witt}}).
\]
For \(M\in\DA\ula(\Hk{X})\),
\begin{equation}\label{eq:reflection-specialization}
    \mcB_{\mcG',\mcG}(i_0^*n_0^*M)
    \simeq
    \Psi_{h_1}\bigl(
        (\rho_\eta^{-1})^*j_0^*n_0^*M
    \bigr).
\end{equation}
If \(\mcG^\prime=\mcG\), the equivalence is
monoidal for convolution.
\end{Theorem}

\begin{proof}
Let $Y_X\subset\Hk{X}$ be a finite closed union of Schubert strata. We
construct by induction on the number of strata an equivalence
\[
    \Theta_Y:\DA\cons(Y_0)\isomto\DA\cons(Y_1)
\]
together with, for every $M\in\DA\ula(Y_X)$, an equivalence
\begin{equation}\label{eq:bounded-specialization-compatibility}
    \epsilon_{Y,M}:
    \Theta_Y(i_0^*n_0^*M)
    \isomto
    \Psi_{Y,1}\bigl(
        (\rho_{\eta,Y}^{-1})^*j_0^*n_0^*M
    \bigr).
\end{equation}
If $Y_X$ consists of one stratum, both assertions follow from Proposition
\ref{proposition:comparison-single-stratum}.

Suppose that $Y_X$ contains more than one stratum. Let
$u_X:U_X\hookrightarrow Y_X$ be an open stratum, indexed by $w$, and let
$z_X:Z_X\hookrightarrow Y_X$ be its closed complement. We denote the
corresponding immersions on $Y_a$ by $u_a$ and $z_a$. Let
\[
    K_a=z_a^*u_{a,*}:
    \DA\cons(U_a)\longrightarrow\DA\cons(Z_a),
    \qquad a=0,1.
\]
By induction, $\Theta_Z$ and $\epsilon_{Z,-}$ have already been
constructed. We compare the gluing functors $K_0$ and $K_1$.

Let $A\in\DA\cons(U_0)$. Under the constant classifying stack description
in the proof of Proposition \ref{proposition:comparison-single-stratum},
let $A_X$ be its pullback along $X\to\Spec(k)$. Set
\[
    G_{Z,X}(A)=z_X^*u_{X,*}A_X.
\]
Note that over $k$, constructible and ULA motives agree by
\cite[Lemma 3.3.12 and Remark 3.3.13]{preisUlaMotives}. In particular, Corollary
\ref{corollary:closed} shows that \(G_{Z,X}(A)\) is ULA over $X$ and
gives
\[
    i_0^*n_0^*G_{Z,X}(A)\simeq K_0(A).
\]
Let $A_{S_1}$ be the pullback of $\theta_w(A)$ along
$S_1\to\Spec(k)$. The definition of $\theta_w$ gives
\[
    i_1^*A_{S_1}\simeq\theta_w(A),
    \qquad
    j_1^*A_{S_1}\simeq
    (\rho_{\eta,U}^{-1})^*j_0^*n_0^*A_X.
\]
Set
\[
    G_{Z,1}(A)=z_{S_1}^*u_{S_1,*}A_{S_1}.
\]
Again by Corollary \ref{corollary:closed}, this object is ULA over $S_1$ and
\[
    i_1^*G_{Z,1}(A)\simeq K_1\theta_w(A).
\]
Since $\rho_\eta$ respects the open and closed strata, base change gives
\[
    j_1^*G_{Z,1}(A)\simeq
    (\rho_{\eta,Z}^{-1})^*j_0^*n_0^*G_{Z,X}(A).
\]
Applying $\epsilon_{Z,-}$ to $G_{Z,X}(A)$ and using
\eqref{eq:ula-nearby-restriction} for $G_{Z,1}(A)$, we obtain a functorial
equivalence
\begin{equation}\label{eq:inductive-gluing-comparison}
\begin{split}
    \lambda_{Y,A}:\Theta_ZK_0(A)
    &\simeq
    \Psi_{Z,1}\bigl(
        (\rho_{\eta,Z}^{-1})^*j_0^*n_0^*G_{Z,X}(A)
    \bigr)\\
    &\simeq K_1\theta_w(A).
\end{split}
\end{equation}
Thus, \eqref{eq:inductive-gluing-comparison} gives an equivalence of functors
\[
    \lambda_Y:\Theta_Z\circ K_0\isomto K_1\circ\theta_w.
\]

Recollement describes an object of $\DA\cons(Y_0)$ by a triple
\[
    (A,B,\delta),
    \qquad
    \delta:B\longrightarrow K_0(A),
\]
see \cite[Remark 3.3.15(5)]{preisUlaMotives}.
Define
\begin{equation}\label{eq:inductive-definition-special-fiber-comparison}
    \Theta_Y(A,B,\delta)
    =
    \bigl(
        \theta_w(A),\,
        \Theta_Z(B),\,
        \lambda_{Y,A}\circ\Theta_Z(\delta)
    \bigr).
\end{equation}
Since $\theta_w$ and $\Theta_Z$ are equivalences and $\lambda_Y$
identifies the two gluing functors, $\Theta_Y$ is an equivalence.

To construct $\epsilon_{Y,M}$, apply $\Psi_{Y,1}$ to the image of
the generic localization triangle under $(\rho_{\eta,Y}^{-1})^*$. Smooth base
change on the open stratum and proper base change on the closed complement
identify its restrictions with the nearby cycles terms for $U_X$ and
$Z_X$, respectively. Proposition
\ref{proposition:comparison-single-stratum} identifies the first term, and
the induction hypothesis identifies the second. The equivalence $\lambda_Y$
identifies the gluing morphisms. Hence the resulting recollement triple is
$\Theta_Y(i_0^*n_0^*M)$, which gives the desired equivalence
$\epsilon_{Y,M}$.

Passing to the filtered
colimit yields an equivalence
\begin{equation}\label{eq:ula-main-equivalence}
    \mcB_{\mcG',\mcG}:
    \DA\cons(\Hk{k}^{\mathrm{eq}})
    \isomto
    \DA\cons(\Hk{k}^{\mathrm{Witt}}).
\end{equation}
Equation \eqref{eq:bounded-specialization-compatibility} shows
\eqref{eq:reflection-specialization}.

It remains to prove the compatibility with convolution, which is treated in Subsection
\ref{subsection:monoidality}.
\end{proof}

\subsection{Monoidality of the equivalence}\label{ss: monoidality}
\label{subsection:monoidality}

We retain the notation of Subsection
\ref{subsection:nearby-comparison}. Assume that
$\mcG^\prime=\mcG$, and write
\[
  \mcB=\mcB_{\mcG,\mcG}:
  \DA\cons(\Hk{}^{\mathrm{eq}})
  \xrightarrow{\sim}
  \DA\cons(\Hk{}^{\mathrm{Witt}}).
\]
Let
\[
  \mathcal E^\otimes
  =
  \operatorname{Fun}\!\left(
    \DA\cons(\Hk{}^{\mathrm{eq}}),
    \DA\cons(\Hk{}^{\mathrm{Witt}})
  \right)^\otimes
  \longrightarrow
  \operatorname{Assoc}^\otimes
\]
be Lurie's functor $\infty$-operad associated with the two
convolution-monoidal categories
\cite[Construction~2.2.6.7 and Example~2.2.6.10]{lurieHA}.
By \cite[Definition~2.1.3.1 and Example~2.2.6.10]{lurieHA}, a lax
convolution-monoidal structure on $\mcB$ is equivalently a section
\[
  s:\operatorname{Assoc}^{\otimes}
  \longrightarrow
  \mathcal E^{\otimes}
\]
of the projection
$\mathcal E^{\otimes}\to\operatorname{Assoc}^{\otimes}$ whose value at
$\langle 1\rangle$ is $\mcB$. Since $s$ is required to be a map of
$\infty$-operads, its value at $\langle r\rangle$ is canonically
equivalent to $\mcB\times\ldots\times\mcB$, and it carries inert
morphisms to inert morphisms. These data are determined, up to a
contractible choice, by its value $\mcB$ at
$\langle1\rangle$. It therefore remains to construct its values on
active morphisms and to verify the compatibilities with composition
and the unit. We will then show that all the resulting structure
morphisms are equivalences, which by
\cite[Definition~2.1.3.7]{lurieHA}, means that the resulting lax
monoidal structure is in fact monoidal.

Fix an active morphism
\[
  \alpha:\langle r\rangle\longrightarrow\langle1\rangle
\]
in $\operatorname{Assoc}^{\otimes}$, and let
$\star_{0,\alpha}$ and $\star_{1,\alpha}$ denote the corresponding
$r$-fold convolution functors. Put
\[
  F_\alpha
  =
  \star_{1,\alpha}
  \circ
  (\mcB\times\cdots\times\mcB),
  \qquad
  G_\alpha
  =
  \mcB\circ\star_{0,\alpha}.
\]
By \cite[Corollary~2.2.6.12]{lurieHA}, the space of operations in
$\mathcal E^\otimes$ lying over $\alpha$, from
$\mcB\times\cdots\times\mcB$ to $\mcB$, is
\[
  \operatorname{Map}_{\operatorname{Fun}
    (\DA\cons(\Hk{}^{\mathrm{eq}})^r,
     \DA\cons(\Hk{}^{\mathrm{Witt}}))}
  (F_\alpha,G_\alpha).
\]

We now construct natural transformations
\begin{equation}
  \mu^{\alpha}:F_\alpha\longrightarrow G_\alpha
  \label{eq:monoidal-operations}
\end{equation}
compatible with composition and the unit.

As in the proof of Theorem \ref{Theorem: Main}, the construction
follows from a recollement argument, recorded in the lemma below.

\begin{Lemma}
\label{lem:multilinear-recollement}
Let $Y_X\subset\Hk{X}$ be a finite closed union of Schubert strata,
and write
\[
  Y_X=U_X\sqcup Z_X
\]
with $U_X$ open. Denote the corresponding decomposition of its
equal-characteristic special fibre by
\[
  Y_0=U_0\sqcup Z_0
\]
and consider the recollement
\[
  \DA\cons(Z_0)
  \mathop{\longrightarrow}^{z_{0,*}}
  \DA\cons(Y_0)
  \mathop{\longrightarrow}^{u_0^*}
  \DA\cons(U_0).
\]
Put
\[
  K_0=z_0^*u_{0,*}:
  \DA\cons(U_0)\longrightarrow
  \DA\cons(Z_0),
\]
and denote by
\begin{equation}
  \partial_{Y,0}:z_{0,*}K_0\longrightarrow u_{0,!}[1]
  \label{eq:universal-boundary}
\end{equation}
the boundary morphism in the localization triangle.

Let $\mathcal A$ be a stable $\infty$-category, and let
\[
  F,G:
  \DA\cons(Y_0)\times\mathcal A
  \longrightarrow
  \DA\cons(\Hk{}^{\mathrm{Witt}})
\]
be functors that are exact in each variable. For $H\in\{F,G\}$, put
\[
  H_U=H(u_{0,!}\times\id),
  \qquad
  H_Z=H(z_{0,*}\times\id),
\]
and let
\[
  \lambda_H:
  H_Z(K_0\times\id)
  \longrightarrow
  H_U[1]
\]
be the natural transformation induced by
$\partial_{Y,0}:z_{0,*}K_0\to u_{0,!}[1]$.
Then restriction to the open and closed parts exhibits the following
square as a pullback:
\begin{equation}
\begin{tikzcd}[column sep=large,row sep=large]
  \operatorname{Map}(F,G)
    \arrow[r]
    \arrow[d]
  &
  \operatorname{Map}(F_U,G_U)
    \arrow[d]
  \\
  \operatorname{Map}(F_Z,G_Z)
    \arrow[r]
  &
  \operatorname{Map}\bigl(
    F_Z(K_0\times\id),G_U[1]
  \bigr).
\end{tikzcd}
\label{eq:recollement-mapping-space}
\end{equation}
The right vertical map sends a point $\eta_U$ to
\[
  \eta_U[1]\circ\lambda_F,
\]
while the lower horizontal map sends $\eta_Z$ to
\[
  \lambda_G\circ
  \eta_Z(K_0\times\id).
\]
The analogous assertion holds after recollement in any finite
collection of source variables. These pullback squares are further
compatible with composition of functors exact in each variable.
\end{Lemma}

\begin{proof}
By stable recollement
\cite[Proposition~A.8.11 and Remark~A.8.18]{lurieHA},
$\DA\cons(Y_0)$ is reconstructed from
$\DA\cons(U_0)$, $\DA\cons(Z_0)$, and the gluing functor
$K_0=z_0^*u_{0,*}$. More concretely, applying an exact functor $H$
to the functorial localization triangle
\[
  u_{0,!}u_0^*M
  \longrightarrow M
  \longrightarrow z_{0,*}z_0^*M
  \longrightarrow u_{0,!}u_0^*M[1]
\]
shows that $H$ is determined by $H_U$, $H_Z$, and
$\lambda_H=H(\partial_{Y,0}\times\id)$. Conversely, these data
reconstruct $H$ by
\begin{equation}
  (M,A')
  \longmapsto
  \operatorname{fib}\Bigl(
    H_Z(z_0^*M,A')
    \longrightarrow H_Z(K_0u_0^*M,A')
    \mathop{\longrightarrow}^{\lambda_H}
      H_U(u_0^*M,A')[1]
  \Bigr).
  \label{eq:functor-reconstruction}
\end{equation}
The localization triangles show that these constructions are inverse.
A transformation $F\to G$ is therefore the same as transformations on
the open and closed parts together with a homotopy between their images
in the lower-right corner of
\eqref{eq:recollement-mapping-space}. This proves that the square is a
pullback. Applying the same argument inductively yields the
several-variable assertion.

Compatibility with composition follows directly from the construction.
Indeed, suppose
\[
  T=H\circ(L_1,\ldots,L_m)
\]
and the glued variable occurs among the inputs of $L_j$. The boundary
transformation of $T$ is obtained by applying $H$ to that of $L_j$:
\[
\begin{aligned}
  T_Z\circ(K_0\times\id)
  &\simeq
  H\bigl(
    \ldots,
    (L_j)_Z\circ(K_0\times\id),
    \ldots
  \bigr)
  \longrightarrow
  H\bigl(\ldots,(L_j)_U[1],\ldots\bigr)
  \simeq T_U[1].
\end{aligned}
\]
The several-variable assertion follows by applying the one-variable
argument inductively, on the number of glued source categories. Since the one-variable
pullback square is natural under composition, so is the several-variable version.
\end{proof}

As in Subsection \ref{subsection:nearby-comparison}, write
\[
  e_0^*=i_0^*n_0^*:
  \DA\ula(\Hk{X})
  \longrightarrow
  \DA\cons(\Hk{}^{\mathrm{eq}})
\]
for restriction to the equal-characteristic special fiber. For
$M\in\DA\ula(\Hk{X})$, let
\begin{equation}
  \epsilon_M:
  \mcB(e_0^*M)
  \xrightarrow{\sim}
  \Psi_{h_1}\!\left(
    (\rho_\eta^{-1})^*j_0^*n_0^*M
  \right)
  \label{eq:ula-comparison}
\end{equation}
denote the equivalence of \eqref{eq:reflection-specialization}.

Let $\star_{X,\alpha}$ denote the ordered iterated convolution over $X$
associated with $\alpha$. By \eqref{eq:ula-nearby-restriction}, for
$E\in\DA\ula(\Hk{S_1})$ there is a natural equivalence
\[
  \Psi_{h_1}(j_1^*E)
  \xrightarrow{\sim}
  i_1^*E.
\]
By Lemma \ref{lemma:convolutionbasechange}, restriction is
convolution-monoidal, while Proposition
\ref{prop: generic equivalence} shows that
\[
  (\rho_\eta^{-1})^*:
  \DA\ula(\Hk{S_{0,\eta}})
  \xrightarrow{\sim}
  \DA\ula(\Hk{S_{1,\eta}})
\]
is convolution-monoidal. Consequently, under the identification above,
the composite
\[
  M\longmapsto
  \Psi_{h_1}\!\left(
    (\rho_\eta^{-1})^*j_0^*n_0^*M
  \right)
\]
is convolution-monoidal. Transporting these structure morphisms through the equivalence $\epsilon$
of \eqref{eq:ula-comparison} gives, for
$M_1,\ldots,M_r\in\DA\ula(\Hk{X})$, the following natural equivalence:
\begin{equation}
\begin{split}
  \gamma^\alpha_{M_1,\ldots,M_r}:\quad
  &\star_{1,\alpha}\bigl(
    \mcB(e_0^*M_1),\ldots,\mcB(e_0^*M_r)
  \bigr)
  \\
  &\xrightarrow{\ \star_{1,\alpha}
    (\epsilon_{M_1},\ldots,\epsilon_{M_r})\ }
  \star_{1,\alpha}\left(
    \Psi_{h_1}\!\left((\rho_\eta^{-1})^*j_0^*n_0^*M_1\right),
    \ldots,
    \Psi_{h_1}\!\left((\rho_\eta^{-1})^*j_0^*n_0^*M_r\right)
  \right)
  \\
  &\xrightarrow{\sim}
  \Psi_{h_1}\!\left(
    (\rho_\eta^{-1})^*j_0^*n_0^*
    \star_{X,\alpha}(M_1,\ldots,M_r)
  \right)
  \\
  &\xrightarrow{\ \epsilon_{\star_{X,\alpha}
      (M_1,\ldots,M_r)}^{-1}\ }
  \mcB\bigl(
    e_0^*\star_{X,\alpha}(M_1,\ldots,M_r)
  \bigr)
  \\
  &\xrightarrow{\sim}
  \mcB\bigl(
    \star_{0,\alpha}(e_0^*M_1,\ldots,e_0^*M_r)
  \bigr).
\end{split}
\label{eq:ula-monoidal-comparison}
\end{equation}
For $r=0$, writing $\delta_X$, $\delta_0$, and
$\delta_1$ for the respective units of convolution, the corresponding
unit morphism is
\begin{equation}
  \delta_1
  \xrightarrow{\sim}
  \Psi_{h_1}\!\left(
    (\rho_\eta^{-1})^*j_0^*n_0^*\delta_X
  \right)
  \xrightarrow{\ \epsilon_{\delta_X}^{-1}\ }
  \mcB(e_0^*\delta_X)
  \xrightarrow{\sim}
  \mcB(\delta_0).
  \label{eq:nullary-comparison}
\end{equation}
The first and last arrows are the unit equivalences of the
convolution-monoidal functors above. Since their structure morphisms
satisfy the unit and composition compatibilities, the same is true of
the morphisms \eqref{eq:ula-monoidal-comparison} and
\eqref{eq:nullary-comparison} obtained by transport through $\epsilon$.

We now extend these morphisms to arbitrary constructible motives by
induction on their supports. Let
$c_X:C_X\hookrightarrow\Hk{X}$ be a Schubert stratum, with special fibre
$c_0:C_0\hookrightarrow\Hk{}^{\mathrm{eq}}$. By the constant
classifying-stack description used in the proof of Proposition
\ref{proposition:comparison-single-stratum}, every
$A\in\DA\cons(C_0)$ has a functorial ULA lift
\[
  A_X\in\DA\ula(C_X),
  \qquad
  i_0^*n_0^*A_X\simeq A.
\]
By Corollary \ref{corollary:closed}, $c_{X,!}A_X$ is ULA and
\[
  e_0^*c_{X,!}A_X\simeq c_{0,!}A.
\]
On objects whose support is a single Schubert stratum, we define $\mu^\alpha$ by evaluating
\eqref{eq:ula-monoidal-comparison} on these extensions by zero.

We will proceed by induction on the number of Schubert cells in supports. Without loss of generality, we may
assume that only the first input is supported on multiple cells.
Suppose that the support of this object is decomposed as
\[
  Y_X=U_X\sqcup Z_X,
\]
where $U_X$ is open, and assume that $\mu^\alpha$ has already been
constructed when the first input is supported on $U_0$ or on $Z_0$.
For $A\in\DA\cons(U_0)$, let
\[
  A_X\in\DA\ula(U_X),
  \qquad
  i_0^*n_0^*A_X\simeq A,
\]
be its constant lift. As in the proof of Theorem
\ref{Theorem: Main}, put
\[
  G_{Z,X}(A)=z_X^*u_{X,*}A_X.
\]
By Corollary \ref{corollary:closed}, $G_{Z,X}(A)$ is ULA. The
localization triangle on $Y_X$ gives a natural morphism
\begin{equation}
  b_A:
  z_{X,*}G_{Z,X}(A)
  \longrightarrow
  u_{X,!}A_X[1].
  \label{eq:lifted-boundary}
\end{equation}
Base change identifies its restriction to the equal-characteristic
special fibre with the boundary morphism
\eqref{eq:universal-boundary}:
\begin{equation}
  e_0^*G_{Z,X}(A)\simeq K_0(A),
  \qquad
  e_0^*(b_A)=
  \partial_{Y,0}(A):
  z_{0,*}K_0(A)\longrightarrow u_{0,!}A[1].
  \label{eq:boundary-special-fibre}
\end{equation}

For ULA lifts $M_2,\ldots,M_r$ of the remaining inputs, naturality of
\eqref{eq:ula-monoidal-comparison} with respect to $b_A$ gives a
homotopy
\begin{equation}
\begin{split}
  &\gamma^\alpha_{u_{X,!}A_X,M_2,\ldots,M_r}[1]
  \circ
  \star_{1,\alpha}\bigl(
    \mcB(e_0^*b_A),\id,\ldots,\id
  \bigr)
  \\
  &\qquad\simeq
  \mcB\!\left(
    \star_{0,\alpha}(e_0^*b_A,\id,\ldots,\id)
  \right)
  \circ
  \gamma^\alpha_{z_{X,*}G_{Z,X}(A),M_2,\ldots,M_r}.
\end{split}
\label{eq:boundary-homotopy}
\end{equation}
By the induction hypothesis, the two occurrences of $\gamma^\alpha$ in
\eqref{eq:boundary-homotopy} restrict to the previously constructed
operations on $U_0$ and $Z_0$. In view of
\eqref{eq:boundary-special-fibre}, the homotopy
\eqref{eq:boundary-homotopy} identifies their images in the lower-right
corner of the pullback square
\eqref{eq:recollement-mapping-space}. Lemma
\ref{lem:multilinear-recollement} therefore determines
\[
  \mu^\alpha:
  \left.F_\alpha\right|_{\DA\cons(Y_0)\times\mathcal A}
  \longrightarrow
  \left.G_\alpha\right|_{\DA\cons(Y_0)\times\mathcal A},
\]
where $\mathcal A$ is the product of the remaining source categories.

Applying this argument successively in the remaining source variables
constructs $\mu^\alpha$ for every tuple of input objects. Since
the pullback square of Lemma \ref{lem:multilinear-recollement} is compatible with composition, the same induction constructs the unit and
composition homotopies from those satisfied by
\eqref{eq:ula-monoidal-comparison} and
\eqref{eq:nullary-comparison}. Thus, for every active morphism $\alpha$,
we obtain compatible natural transformations
\begin{equation}
  \mu^\alpha:F_\alpha\longrightarrow G_\alpha.
  \label{eq:bounded-structure-maps}
\end{equation}

Further, note that each $\mu^\alpha$ is an equivalence. Indeed, this holds on input objects
supported on a single stratum by \eqref{eq:ula-monoidal-comparison}, and it is clear that our
inductive recollement preserves equivalences. For the unit it
follows from \eqref{eq:nullary-comparison}.

Together with the inert morphisms fixed above, the transformations
$\mu^\alpha$ therefore define a section
\[
  s:\operatorname{Assoc}^{\otimes}
  \longrightarrow
  \mathcal E^{\otimes}
\]
and hence a lax convolution-monoidal structure on $\mcB$ by
\cite[Example~2.2.6.10]{lurieHA}. Since every structure morphism
$\mu^\alpha$ is an equivalence,
\cite[Definition~2.1.3.7]{lurieHA} shows that the functor $\mathcal{B}$ is
convolution-monoidal.

\section{Compatibility with monoidal structures on étale Satake categories}\label{Section: Compatibility with symmetric monoidal structures}
In this section we work over $k=\overline{k}_{F}.$
Let $G$ be as in Situation \ref{Situation: tamely ramified group} and we continue to assume that $G=G^{*}$ is quasi-split.
Then, for $R\in \Perf_{k},$ $B(R)$ is an $\breve{\mcO}:=\mcO_{\breve{F}}$-algebra.
Let $x_{0}\in \mcA(G,S;\breve{F})$ be
\emph{special} (see \cite[Definition 7.11.1]{KalethaPrasadBruhatTits}); recall that such a point is automatically a vertex.
We obtain the Pappas--Zhu group scheme $\mcG:=\mcG_{x_{0}}$ and we write $\widetilde{W}_{x_{0}}\subseteq \widetilde{W}$ for the associated subgroup of the Iwahori--Weyl group.
When considering Hecke stacks we will always assume $\mcG^{\prime}=\mcG:=\mcG_{x_{0}}$ in this section.

Let $\ell\neq p$ be a prime and $\Lambda\in \lbrace \mbF,L,\mcO_{L},\overline{\mathbb{Q}}_{\ell} \rbrace$, where $L/\mathbb{Q}_{\ell}$ is an algebraic extension, $\mcO_{L}$ is its ring of integers and $\mbF$ is the residue field of $\mcO_{L}$.
Let $T$ be a perfect $\Divtilk$-scheme.
By Proposition \ref{Proposition: representability Bando affine Grassmannian} and Lemma \ref{Lemma: Bando affine Grassmannian ind-projective for Pappas--Zhu group schemes}, we may write $\Gr^{B}_{\mcG,T}=\colim_{i\in I}X_{i}$ as a colimit of perfections of projective $T$-schemes, perfectly of finite presentation along perfectly finitely presented closed immersions such that $X_{i}$ is $L^{+}\mcG$-invariant.
Exactly as in Section \ref{subsection: bounded and ula motives on Hecke stacks}, we may define the category of \emph{bounded} étale sheaves on the Hecke stack as $$D(\Hk{\mcG,T},\Lambda)^{\bd}=\colim_{i\in I} D_{L^{+}\mcG}(X_{i},\Lambda).$$
As in Section \ref{subsection: Convolution}, we may define the \emph{convolution} product
$$
\star_{T}\colon D(\Hk{\mcG,T},\Lambda) \otimes_{\mcA_{T}^{\et}} D(\Hk{\mcG,T},\Lambda)\rightarrow D(\Hk{\mcG,T},\Lambda).
$$
The convolution product commutes with pullback along $\Divtilk$-scheme morphisms $T^{\prime}\rightarrow T$, preserves $D(\Hk{\mcG,-},\Lambda)^{\bd}$ and induces a monoidal structure on the $\infty$-category $D(\Hk{\mcG,-},\Lambda)^{\bd}$.

Recall that the notion of ULA-ness for étale sheaves is defined as for étale motives (see e.g. \cite[Definition 3.2]{HansenScholze}).
Following Bando, we also use Hansen-Scholze's notion of \emph{relative perversity} (\cite{HansenScholze}) to define the notion of a \emph{perverse sheaf} $\mcF\in D(\Hk{\mcG,T},\Lambda)^{\bd}$ as in \cite[Definition 6.2]{bandoComparison} and we use the same definition of a \emph{flat perverse sheaf} as in \cite[Definition 6.3]{bandoComparison}.
Then we let $\Sat(\Hk{\mcG,T},\Lambda)\subset D{\ula}(\Hk{\mcG,T},\Lambda)^{\bd}$ be the full subcategory of flat perverse sheaves; it is called the \emph{Satake category}.
Note that the convolution product on 
$D(\Hk{\mcG,T},\Lambda)^{\bd}$ induces a monoidal structure on $\Sat(\Hk{\mcG,T},\Lambda)$ seen inside the homotopy category of $D(\Hk{\mcG,T},\Lambda)^{\bd}$.
\begin{proposition}\label{proposition: Satake categories get identified}
    The equivalence in Theorem \ref{Theorem: Main} induces upon passing to étale realizations a monoidal equivalence
    $$
    \Sat(\Hk{\mcG,k}^{\eq},\Lambda)\simeq \Sat(\Hk{\mcG,k}^{\Witt},\Lambda).
    $$
\end{proposition}
\begin{Remark}
    Both Satake categories in Proposition \ref{proposition: Satake categories get identified} carry fiber functors in the form of total cohomology (\cite{ModularRamifiedSatake} and \cite{vandenHoveRamifiedMotivicSatake}).
    The proof we are about to present moreover identifies these fiber functors.
    Consequently, under the imported symmetry constraints, the equivalence is \emph{symmetric} monoidal.
\end{Remark}
For the proof, we need to discuss \emph{constant-term functors} in our setting. 

For this, fix a rigidification $(\underline{G},\underline{A}=\underline{S},\underline{P})$.
The split torus $\underline{S}$ over $\mcO[u^{\pm}]$ extends to a split torus $\mcS$ over $\mcO[u]$.
A given cocharacter $\lambda\colon \mathbb{G}_{m,\mcO[u^{\pm}]}\rightarrow \underline{S}$ then extends to a cocharacter of $\mcS$ because $\Spec(\mcO[u])$ is connected and $\mcS$ is split.
Composing with the inclusion $\mcS\subset \mcT \subset \mcG$, we obtain a cocharacter $\lambda$ of $\mcG$.
Definitions of the following objects can be found in \cite[Section 1.1]{RicharzGmAction}.
We obtain the attractor resp. repeller locus $\mcG^{+}$ ($\underline{G}^{+}$) resp. $\mcG^{-}$ ($\underline{G}^{-}$) and the fixed point locus $\mcG^{0}$ ($\underline{G}^{0}$).
Observe that we have natural limit morphisms $q^{\pm}\colon \mcG^{\pm}\rightarrow \mcG^{0}$ and we denote the inclusions by $i^{\pm}\colon \mcG^{\pm}\rightarrow  \mcG$.
We summarize the basic properties in the following statement.
\begin{Lemma}
    \begin{enumerate}
    \item[(a)] $\mcG^{0}$, $\mcG^{\pm}$ are smooth closed $\mcO[u]$-subgroup schemes of $\mcG$ with connected fibers, and after inverting $u$ they agree with $\underline{G}^{0}$ resp. $\underline{G}^{\pm}$,
    \item[(b)] for $C\in \Perf_{k_{F}}$ an algebraically closed field and all $(\xi_{1},\xi_{2})\in \Divtilk(C)$, $\mcG^{0}_{B^{+}_{(\xi_{1},\xi_{2})}(C)}$ is a parahoric subgroup scheme of $\underline{G}^{0}_{B_{(\xi_{1},\xi_{2})}(C)}$,
    \item[(c)] $\mcG^{\pm}=\mcG^{0}\ltimes \mcN^{\pm}$, where $\mcN^{\pm}$ is smooth affine with geometrically connected fibers.
    \end{enumerate}
\end{Lemma}
\begin{proof}
    This follows as in \cite[Lemma 4.5]{HainesRicharzJAMS}, using the arguments in the proof of Theorem \ref{Theorem: PZ} together with the fact that attractor, repeller and fixed point loci commute with base change. 
\end{proof}
Consider the composition $\mathbb{G}_{m,k}^{\perf}\rightarrow L^{+}\mathbb{G}_{m}\rightarrow L^{+}\mcG$ given by the Teichm\"uller lift and the cocharacter $\lambda$.
This gives an action of $\mathbb{G}_{m,k}^{\perf}$ on $\Gr^{B}_{\mcG}$, which preserves Schubert varieties.
Consider the attractor resp. repeller subsheaves $(\Gr^{B}_{\mcG})^{\pm}\subset \Gr_{\mcG}^{B}$ and the fixed points subsheaf $(\Gr^{B}_{\mcG})^{0}\subset \Gr^{B}_{\mcG}$.
Note that $(\Gr^{B}_{\mcG})^{\pm}$ is representable by $\Gr^{B}_{\mcG^{\pm}}$ and $(\Gr^{B}_{\mcG})^{0}$ is representable by $\Gr^{B}_{\mcG^{0}}$ since we are assuming that $\mcG$ is constructed from a \emph{special} parahoric (namely one may argue the same way as in \cite[Proposition 3.9]{vandenHoveRamifiedMotivicSatake}).

We obtain $i^{\pm}\colon \Gr^{B}_{\mcG^{\pm}}\rightarrow \Gr^{B}_{\mcG}$ and $p^{+}\colon \Gr^{B}_{\mcG^{+}}\rightarrow \Gr^{B}_{\mcG^{0}}$ resp. $p^{-}\colon \Gr^{B}_{\mcG^{-}}\rightarrow \Gr^{B}_{\mcG^{0}}$.
These maps are $L^{+}\mcG^{0}$-equivariant.
We keep on denoting by $p^{\pm}$ resp. $i^{\pm}$ the maps obtained by passing to the quotient by $L^{+}\mcG^{0}$.
Let us denote by $h_{\mcG^{0},\mcG}\colon [L^{+}\mcG^{0}\backslash \Gr^{B}_{\mcG,T}]\rightarrow \Hk{\mcG,T}$ the natural map.
We obtain the \emph{naive} constant term functor
$$
\CT_{\mcG^{+},\mcG}^{\naive}\colon D(\Hk{\mcG,T},\Lambda)^{\bd}\rightarrow D(\Hk{\mcG^{0},T},\Lambda)^{\bd}
$$
given by $\mcF\mapsto p^{+}_{!}i^{+,*}h^{*}_{\mcG^{0},\mcG}(\mcF)$ (all functors are derived).
We collect properties of this construction that mirror classically known properties of the constant term functor.
\begin{Lemma}\label{lem: properties of naive constant term}
In the above situation we have that:
    \begin{enumerate}
        \item[(a)] $\CT^{\naive}_{\mcG^{+},\mcG}(\mcF)\simeq p^{-}_{*}i^{-,!}h^{*}_{\mcG^{0},\mcG}(\mcF)$ and $\CT^{\naive}_{\mcG^{+},\mcG}$ commutes with any base change in $T$,
        \item[(b)] If $\lambda$ is a regular cocharacter, then $\CT^{\naive}_{\mcG^{+},\mcG}$ is conservative.
    \end{enumerate}
\end{Lemma}
\begin{proof}
Item (a) follows from Braden's localization theorem in our setting, see \cite[Theorem B]{RicharzGmAction} (see also for a more detailed deduction that works the same way here \cite[Proposition 4.4]{ModularRamifiedSatake}).
For item (b) the same argument as the one given for \cite[Proposition 4.6]{ModularRamifiedSatake} works.
Namely, since $\lambda$ is regular, we know that $\underline{G}^{0}=\underline{T}$ for a maximal tamely ramified torus $\underline{T}$ of $\underline{G}$.
Let us write $\mcG^{0}=\mcT$.
Let $\mcF\in D(\Hk{\mcG,T},\Lambda)^{\bd}$ be non-zero.
Then there must exist $C\in \Perf_{k_{F}}$ algebraically closed and $\Spec(C)\rightarrow T$, such that $\mcF\restriction_{\Hk{\mcG,C}}$ is non-zero.
We also obtain $(\xi_{1},\xi_{2})\in \Divtilk(C)$ and write $B(C)=B_{(\xi_{1},\xi_{2})}(C)$ resp. $B^{+}(C)=B^{+}_{(\xi_{1},\xi_{2})}(C)$.
There exists $\lambda^{\prime}\in X_{*}(\underline{T})^{+}_{\gamma_{0}} \subset X_{*}(\underline{T})_{\gamma_{0}}=\underline{T}(B(C))/\mcT(B^{+}(C))$ such that $\Gr^{\lambda^{\prime},B}_{\mcG,C}=L^{+}_{C}\mcG\cdot t^{\lambda,\prime}\subset \Gr^{B}_{\mcG,C}$ is open in the support of $\mcF\restriction_{\Hk{\mcG,C}}$. Here $t^{\lambda,\prime}\in \underline{T}(B(C))$ is a lift of $\lambda^{\prime}$ and $X_{*}(\underline{T})^{+}_{\gamma_{0}}$ denotes a system of representatives for the quotient of $X_{*}(\underline{T})_{\gamma_{0}}$ by the finite Weyl group $W_{0}$ cf. \cite[Lemma 2.6]{ModularRamifiedSatake}.
Let $\mu$ be the unique $\widetilde{W}_{x_{0}}$ conjugate of $\lambda^{\prime}$ that belongs to $-X_{*}(\underline{T})^{+}_{\gamma_{0}}$.
Then we have for the semi-infinite orbit $S_{\mu,C}$ (same definition as in \cite[Section 3.4]{ModularRamifiedSatake}) that $|S_{\mu,C}\cap \Gr^{\lambda^{\prime},B}_{\mcG,C}|=\lbrace t^{\mu} \rbrace$ by the same proof as \cite[Lemma 5.3]{AGLR}.
Therefore, $\CT^{\naive}(\mcF)\neq 0$ since it is not zero on the component of $\Gr^{B}_{\mcT,C}$ corresponding to $\mu$.
\end{proof}
As usual, we have to shift the naive constant term functor slightly to obtain a functor that is t-exact for the perverse t-structure.
For this, let $\rho\in X^{*}(\underline{T})$ be the half-sum of positive roots of $\underline{G}$ (this identifies with the one for $H$) and similar let $\rho_{G^{0}}$ be the half-sum of positive roots of $\underline{G}^{0}$.
Recall that we have constructed in Definition \ref{def: Kottwitz map on Bando aff Grass} a locally constant map
$$
\kappa_{\underline{G}^{0}}\colon |\Gr^{B}_{\mcG^{0}}|\longrightarrow \pi_{1}(G^{0})_{I}
$$
and pairing with $2\rho-2\rho_{G^{0}}$ (note that this is orthogonal to the coroots of $\underline{G}^{0}$ and $I$-invariant), we obtain a locally constant map
$$
\corr_{\mcG^{0},\mcG}\colon |\Gr^{B}_{\mcG^{0}}|\longrightarrow \mbZ.
$$
\begin{Definition}[Constant term functor]
    The functor
    $$
    \CT_{\mcG^{+},\mcG}\colon D(\Hk{\mcG,T},\Lambda)^{\bd}\rightarrow D(\Hk{\mcG^{0},T},\Lambda)^{\bd}
    $$
    is given by $\CT_{\mcG^{+},\mcG}^{\naive}[\corr_{\mcG^{0},\mcG}]$.
\end{Definition}
\begin{Lemma}\label{Lemma: recoginzing Satake via constant term}
    Let $\mcF\in D{\ula}(\Hk{\mcG,T},\Lambda)^{\bd}$.
    Assume that $\lambda$ is regular and consider the constant term $\CT_{\mcG^{+},\mcG}$ functor just defined.
    We have that
    $\mcF\in \Sat(\Hk{\mcG,T},\Lambda)$ if and only if $\CT_{\mcG^{+},\mcG}(\mcF)\in \Sat(\Hk{\mcG^{0},T},\Lambda)$.
\end{Lemma}
\begin{proof}
    Since we already know by Lemma \ref{lem: properties of naive constant term} that $\CT_{\mcG^{+},\mcG}$ is conservative, it suffices to show $t$-exactness for the relative perversity.
    By definition, this therefore reduces to the case $T=\Spec(C)$, where $C\in \Perf_{k_{F}}$ is algebraically closed and we are considering all objects over perfect $C$-algebras and we are concerned with the usual notion of perverse sheaves.
    We follow the proof of \cite[Proposition 4.6]{ModularRamifiedSatake}; in particular, as in that reference, we may reduce to the case when $\Lambda$ is a field.
    Therefore, let $\mcF\in ^{p}D^{\leq 0}(\Hk{\mcG,C},\Lambda)$.
    Then $\mcF$ is a successive extension of sheaves of the form $j^{\lambda,\prime}_{!}\Lambda[\langle \lambda^{\prime}, 2\rho \rangle]$, where $j^{\lambda,\prime}\colon \Gr^{\lambda,\prime,B}_{\mcG,C}\hookrightarrow \Gr^{B}_{\mcG,C}$ is the locally closed immersion already considered in the proof of Lemma \ref{lem: properties of naive constant term} above.
    Then we have that
    $$
    (\CT_{\mcG^{+},\mcG}(j^{\lambda,\prime}_{!}\Lambda[\langle \lambda^{\prime}, 2\rho \rangle]))_{t^{\mu}}=R\Gamma_{c}(S_{\mu,C}\cap \Gr^{\lambda,\prime,B}_{\mcG,C},\Lambda)[\langle \lambda^{\prime}+\mu,2\rho \rangle].
    $$
    Vanishing of the complex $R\Gamma_{c}(S_{\mu,C}\cap \Gr^{\lambda,\prime,B}_{\mcG,C},\Lambda)$ in degrees greater than $\langle \lambda^{\prime} + \mu, 2\rho \rangle$ therefore follows from $S_{\mu,C}\cap \Gr^{\lambda,\prime,B}_{\mcG,C}$ being equidimensional of dimension $\langle \lambda^{\prime}+\mu, \rho \rangle$, which can be proved in our setting as in \cite[Lemma 5.5]{AGLR}. This shows right $t$-exactness. Left $t$-exactness follows by Verdier duality and using Lemma \ref{lem: properties of naive constant term}(a).
\end{proof}
Now we can turn to the proof of Proposition \ref{proposition: Satake categories get identified}.
\begin{proof}
    Note that upon passing to $\ell$-adic realizations, the functor that induces the equivalence in Theorem \ref{Theorem: Main} agrees with the obvious extension to the ramified case of the functor that Bando uses in \cite[Theorem 5.1]{bandoComparison}.
    Therefore, given the previous Lemma \ref{Lemma: recoginzing Satake via constant term}, the same proof as the one for \cite[Theorem 6.6]{bandoComparison} applies also here.
\end{proof}
\printbibliography
\end{document}